\documentclass{article}
\usepackage[utf8]{inputenc}

\usepackage[margin=3cm]{geometry}

\usepackage{setspace}
\usepackage{datetime}
\usepackage{fancyhdr}
\usepackage{tikz}
\usepackage{hyperref}

\usepackage{hyperref}
\usepackage{authblk}
\usepackage{amsthm}
\usepackage{amsmath}
\usepackage{amssymb}
\usepackage{amsfonts}
\usepackage{mathtools}
\usepackage{algorithm}
\usepackage{braket}
\usepackage{algpseudocode}
\usepackage{graphicx}
\usepackage{xcolor}
\usepackage{xspace}
\usepackage[graphicx]{realboxes}
\usepackage{hypbmsec}
\usepackage{mathtools}

\newtheorem{observation}{Observation}
\newtheorem{theorem}{Theorem}
\newtheorem{lemma}{Lemma}
\newtheorem{proposition}{Proposition}
\newtheorem{definition}{Definition}
\newtheorem{corollary}{Corollary}

\newcommand{\word}[1]{\ensuremath{\bar{#1}}}
\newcommand{\wordcommand}[1]{\word{#1}}

\newcommand{\necklace}[1]{\ensuremath{\tilde{\mathbf{#1}}}}
\newcommand{\Angle}[1]{\ensuremath{\langle #1 \rangle}}

\title{Ranking Binary Unlabelled Necklaces in Polynomial Time}

\author{Duncan Adamson\footnote{Supported by Icelandic Research Fund grant no. 217965} \and ICE-TCS, Department of Computer Science, Reykjavik University}
\begin{document}

\maketitle

\begin{abstract}
    Unlabelled Necklaces are an equivalence class of cyclic words under both the rotation (cyclic shift) and the relabelling operations.
    The relabelling of a word is a bijective mapping from the alphabet to itself.
    The main result of the paper is the first  polynomial-time algorithm for ranking unlabelled necklaces of a binary alphabet.
    The time-complexity of the algorithm is $O(n^6 \log^2 n)$, where $n$ is the length of the considered necklaces.
    The key part of the algorithm is to compute the rank of any word with respect to the set of unlabelled necklaces by finding three other ranks: the rank over all necklaces, the rank over symmetric unlabelled necklaces, and the rank over necklaces with an enclosing labelling. 
    The last two concepts are introduced in this paper.
\end{abstract}

\section{Introduction}


For classes of words under lexicographic (or dictionary) order, a unique integer can be assigned to every word corresponding to the number of words smaller than it.
Such an integer is called the \emph{rank} of a word.
The {\em ranking} problem asks to compute the rank of a given word.
Ranking has been studied for various objects including partitions~\cite{RankPartition}, permutations~\cite{RankPerm1,RankPerm2}, combinations~\cite{RankComb}, etc.

The ranking problem is straightforward for the set of all words over a finite alphabet (assuming the standard lexicographic order), however this ceases to be the case once additional symmetry is introduced.
One such example is combinatorial necklaces \cite{Graham1994}.
A \emph{necklace}, also known as a \emph{cyclic word}, is an equivalence class of all words under the cyclic rotation operation, also known as a cyclic shift.
Necklaces are classical combinatorial objects and they remain an object of study in other contexts such as total search problems~\cite{SplitNeck19} or circular splicing systems~\cite{CircularSplicing}.
The first class of cyclic words to be ranked were \emph{Lyndon words} - fixed length aperiodic cyclic words - by Kociumaka et al. \cite{Kociumaka2014} who provided an $O(n^3)$ time algorithm, where $n$ is the length of the word.
An algorithm for ranking necklaces - fixed length cyclic words - was given by Kopparty et al. \cite{Kopparty2016}, without tight bounds on the complexity.
A quadratic algorithm for ranking necklaces was provided by Sawada et al. \cite{Sawada2017}.
More recently algorithms have been presented for ranking multidimensional necklaces \cite{adamson2021multidimensional} and bracelets \cite{Adamson2021}.

\emph{Our Results}
This paper presents the first polynomial time algorithm for ranking \emph{binary unlabelled necklaces}.
Informally, binary unlabelled necklaces can be though of as necklaces over a binary alphabet with the additional symmetry over the \emph{relabelling} operation, a bijection from the set of symbols to itself.
Considered in terms of binary values, the words $0001$ and $1110$ are equivilent under the relabelling operation, however $1010$ and $1100$ are not.
We provide an $O(n^6 \log^2 n)$ time algorithm for ranking an unlabelled binary necklace within the set of unlabelled binary necklaces of length $n$.

\section{Preliminaries}
\label{sec:prelims}

Let $\Sigma$ be a finite alphabet.
For the remainder of this work we assume $\Sigma$ to be $\{0,1\}$ where $0 < 1$.
We denote by $\Sigma^*$ the set of all words over $\Sigma$ and by $\Sigma^n$ the set of all words of length $n$.
The notation $\word{w}$ is used to clearly denote that the variable $\word{w}$ is a word.
The length of a word $\word{w} \in \Sigma^*$ is denoted by $|\word{w}|$.
We use $\word{w}_i$, for any $i \in \{1,\hdots,|\word{w}|\}$ to denote the $i^{th}$ symbol of $\word{w}$.
Given two words $\word{w},\word{u} \in \Sigma^*$, the \emph{concatenation operation} is denoted by $\word{w} : \word{u}$, returning the word of length $|\word{w}| + |\word{u}|$ where $(\word{w} : \word{u})_i$ equals either $\word{w}_i$, if $i \leq |\word{w}|$ or $\word{u}_{i - |\word{w}|}$ if $i > |\word{w}|$.
The $t^{th}$ power of a word $\word{w}$, denoted by $\word{w}^t$, equals $\word{w}$ repeated $t$ times.

Let $[n]$ be the ordered sequence of integers from $1$ to $n$ inclusive and let $[i,j]$ be the ordered sequence of integers from $i$ to $j$ inclusive.
Given two words $\word{u},\word{v} \in \Sigma^*$,
$\word{u} = \word{v}$ if and only if $|\word{u}| = |\word{v}|$ and $\word{u}_i = \word{v}_i$ for every $i \in [|\word{u}|]$.
A word $\word{u}$ is \emph{lexicographically smaller} than $\word{v}$ if there exists an $i \in [|\word{u}|]$ such that $\word{u}_1 \word{u}_2 \hdots \word{u}_{i-1} = \word{v}_1 \word{v}_2 \hdots \word{v}_{i-1}$ and $\word{u}_i < \word{v}_i$.
Given two words $\word{v},\word{w} \in \Sigma^*$ where $|\word{v}| \neq |\word{w}|$, $\word{v}$ is smaller than $\word{w}$ if $\word{v}^{|\word{w}|} < \word{w}^{|\word{v}|}$ or $\word{v}^{|\word{w}|} = \word{w}^{|\word{v}|}$ and $|\word{v}| < |\word{w}|$.
For a given set of words $\mathbf{S}$, the \emph{rank} of $\word{v}$ with respect to $\mathbf{S}$ is the number of words in $\mathbf{S}$ that are smaller than $\word{v}$.

The \emph{subword} of a cyclic word $\word{w} \in \Sigma^n$ denoted $\word{w}_{[i,j]}$ is the word $\word{u}$ of length $n + j - i + 1 \bmod n$ such that $\word{u}_{a} = \word{w}_{i + a \bmod n}$, i.e. the word such that the $a^{th}$ symbol of $\word{u}$ corresponds to the symbol at position $i + a \bmod n$ of $\word{w}$.
The value of the $t^{th}$ symbol of $\word{w}_{[i,j]}$ is the value of the symbol at position $i + t - 1$ of $\word{w}$.
By this definition, given $\word{u} = \word{w}_{[i,j]}$, the value of $\word{u}_t$ is the $i + t - 1^{th}$ symbol of $\word{w}$ and the length of $\word{u}$ is $|\word{u}| = j - i + 1$.
The notation $\word{u} \sqsubseteq \word{w}$ denotes that $\word{u}$ is a subword of $\word{w}$.
Further, $\word{u} \sqsubseteq_{i} \word{w}$ denotes that $\word{u}$ is a subword of $\word{w}$ of length $i$.

The {\em rotation} of a word $\word{w} \in \Sigma^n$ by $r \in [0,n-1]$ returns the word $\word{w}_{[r + 1,n]}  : \word{w}_{[1,r]}$, and is denoted by $\Angle{ \word{w} }_r$, i.e. $\Angle{ \word{w}_1 \word{w}_2 \hdots \word{w}_n }_r = \word{w}_{r + 1} \hdots \word{w}_n \word{w}_1 \hdots \word{w}_r$.
Under the rotation operation, the word $\word{u}$ is equivalent to the word $\word{v}$ if $\word{v} = \Angle{ \word{u}}_r$ for some $r$.
A word $\word{w}$ is \emph{periodic} if there is a subword $\word{u} \sqsubseteq \word{w}$ and integer $t \geq 2$ such that $\word{u}^t = \word{w}$.
Equivalently, word $\word{w}$ is \emph{periodic} if there exists some rotation $0 < r < |\word{w}|$ where $\word{w} = \Angle{ \word{w}}_r$.
A word is \emph{aperiodic} if it is not periodic.
The \emph{period} of a word $\word{w}$ is the aperiodic word $\word{u}$ such that $\word{w} = \word{u}^t$.

A \emph{necklace} is an equivalence class of words under the rotation operation.
The notation $\necklace{w}$ is used to denote that the variable $\necklace{w}$ is a necklace.
Given a necklace $\necklace{w}$, the \emph{canonical representation} of $\necklace{w}$ is the lexicographically smallest element of the set of words in the equivalence class $\necklace{w}$.
The canonical representation of $\necklace{w}$ is denoted by $\Angle{ \necklace{w} }$, and the $r^{th}$ shift of the canonical representation is denoted by $\Angle{\necklace{w}}_{r}$.
Given a word $\word{w}$, $\Angle{\word{w}}$ denotes the canonical representation of the necklace containing $\word{w}$, i.e. the canonical representation of the necklace $\necklace{u}$ where $\word{w} \in \necklace{u}$.
The set of necklaces of length $n$ over an alphabet of size $q$ is denoted by $\mathcal{N}_q^n$.
Let $\word{w} \in \mathcal{N}_q^n$ denote that the word $\word{w}$ is the canonical representation of some necklace $\necklace{w} \in \mathcal{N}_q^n$.
An aperiodic necklace, known as a \emph{Lyndon word}, is a necklace representing the equivalence class of some aperiodic word.
Note that if a word is aperiodic, then every rotation of the word is also aperiodic.
The set of Lyndon words of length $n$ over an alphabet of size $q$ is denoted by $\mathcal{L}_q^n$.

As both necklaces and Lyndon words are classical objects, there are many fundamental results regarding each objects.
The first results for these objects were equations determining the number of necklaces or Lyndon words of a given length.
The number of (1D) necklaces is given by the equation {\footnotesize $|\mathcal{N}_q^n| = \frac{1}{n} \sum\limits_{d | n} \phi\left( \frac{n}{d}\right)q^d$} where $\phi(n)$ is Euler's totient function \cite{Graham1994}.
Similarly the number of Lyndon words is given with the equation {\footnotesize$|\mathcal{L}_q^n| = \sum\limits_{d | n} \mu\left(\frac{n}{d} \right) |\mathcal{N}_q^d |$}, where $\mu(x)$ is the M\"{o}bius function \cite{Graham1994}.
The \emph{rank} of a word $\word{w}$ in the set of necklaces $\mathcal{N}_q^n$ is the number of necklaces with a canonical representation smaller than $\word{w}$.

\subsection{Unlabelled Necklaces.}

An \emph{unlabelled necklace} is a generalisation of the set of necklaces.
At a high level, two words $\word{v},\word{u} \in \Sigma^n$ belong to the same unlabelled necklace class $\necklace{w}$ if there exists some labelling function $\psi(x) : \Sigma \mapsto \Sigma$ and rotation $r \in [n]$ such that $(\Angle{\word{v}}_r)_i = \psi(\word{u}_{i})$ for every $i \in [n]$.
More formally, let $\psi(x)$ be a bijection from $\Sigma$ into $\Sigma$, i.e. a function taking as input some symbol in $\Sigma$ and returning a symbol in $\Sigma$ such that $\{\psi(x) | \forall x \in \Sigma\} = \Sigma$.
For notation $\psi(\word{w})$ is used to denote the word constructed by applying $\psi(x)$ to every symbol in $\word{w}$ in order, formally $\psi(\word{w}) = \psi(\word{w}_1) \psi(\word{w}_2) \hdots \psi(\word{w}_n)$.
Similarly, the notation $\psi(\necklace{w})$ is used to denote the necklace class constructed by applying $\psi(\word{w})$ to every word $\word{w} \in \necklace{w}$.
Further, let $\Psi(\Sigma)$ be the set of all such functions.
The unlabelled necklace $\necklace{w}$ with a canonical representation $\word{w}$ contains every word $\word{v} \in \Sigma^{n}$ where $\psi(\Angle{\word{v}}_r) = \word{w}$ for some $\psi(x) \in \Psi(\Sigma)$ and $r \in [n]$.
As in the labelled case, the canonical representation of an unlabelled necklace $\necklace{w}$, denoted $\Angle{\necklace{w}}$, is the lexicographically smallest word in the equivalence class.
The set of unlabelled $q$-arry necklaces of length $n$ is denoted $\hat{\mathcal{N}}_q^{n}$, and the set of $q$-arry Lyndon words of length $n$ $\hat{\mathcal{L}}_q^{n}$.

In this paper we study \emph{binary unlabelled necklaces}, in other words unlabelled necklaces restricted to a binary alphabet.
In this case $\Sigma = \{0,1\}$ and $\Psi(\Sigma)$ contains the identity function $I(x)$, where $I(x) = x$, and the swapping function $S(x)$ where $S(x) = \begin{cases}
0 & x = 1\\
1 & x = 0
\end{cases}$.
Gilbert and Riordan \cite{gilbert1961symmetry} provide the following equations for computing the sizes of $\hat{\mathcal{N}}_2^{n}$ and $\hat{\mathcal{L}}_2^{n}$:

\begin{align*}
    |\hat{\mathcal{N}}_2^{n}| = \sum\limits_{\text{odd } d | n}\phi(d) 2^{n/d}\\
    |\hat{\mathcal{L}}_2^{n}| = \sum\limits_{\text{odd } d | n}\mu(d) 2^{n/d}
\end{align*}



\noindent
In this paper we introduce two subclasses of unlabelled necklaces, the class of \emph{symmetric unlabelled necklaces} and the class of \emph{enclosing unlabelled necklaces} for some given word $\word{w}$.
Observe that a binary unlabelled necklace $\necklace{w}$ may correspond to either one or two (labelled) necklaces.
Informally, a symmetric unlabelled necklaces is such an unlabelled necklaces that corresponds to only a single necklace.
An enclosing unlabelled necklace relative to a word $\word{w}$ is a non-symteric unlabelled necklace corresponding to a pair of necklaces $\necklace{v}$ and $\necklace{u}$ such that $\necklace{v} < \word{w} < \necklace{u}$.
Any Lyndon word that is a symmetric unlabelled necklace is a \emph{symmetric unlabelled Lyndon word}, and any unlabelled Lyndon word that encloses a word $\word{w}$ is an \emph{enclosing unlabelled Lyndon word} of $\word{w}$.

\begin{definition}[Symmetric Necklaces]
A binary necklace $\necklace{w}$ is \emph{symmetric} if and only if $\necklace{w} = S(\necklace{w})$.
\end{definition}

\begin{definition}[Enclosing Unlabelled Necklaces]
An unlabelled necklace $\necklace{u}$ \emph{encloses} a word $\word{w}$ if $\Angle{\necklace{u}} < \word{w} < \Angle{S(\necklace{u})}$.
An unlabelled necklace $\necklace{u}$ is an \emph{enclosing unlabelled necklace} of $\word{w}$ if $\necklace{u}$ encloses $\word{w}$.
\end{definition}

\subsection{Bounding Subwords}

One important tool that is used in the ranking of unlabelled necklaces are \emph{bounding subwords}, introduced in \cite{Adamson2021}.
Informally, bounding subwords of length $l \leq n$ provide a means to partition $\Sigma^l$ into $n + 2$ sets based on the subwords of some $\word{w} \in \Sigma^n$ of length $l$.
Given two subwords $\word{v}, \word{u} \sqsubseteq_{l} \word{w}$ such that $\word{v} < \word{u}$ the set $S(\word{v}, \word{u})$ contains all words in $\Sigma^l$ that are between the value of $\word{v}$ and $\word{u}$, formally $S(\word{v}, \word{u}) = \{\word{x} \in \Sigma^l | \word{v} \leq \word{x} < \word{u}\}$.
In this paper we are only interested in sets between pairs $\word{v},\word{u} \sqsubseteq_l \word{w}$ where there exists no $\word{s} \sqsubseteq_l \word{w}$ such that $\word{v} < \word{s} < \word{u}$.
As such, we define a subword of $\word{w}$ as \emph{bounding} some word $\word{v}$ if it is the lexicographically largest subword of $\word{w}$ that is smaller than $\word{v}$.

\begin{definition}[Bounding Subwords]
Let $\wordcommand{w}, \wordcommand{v} \in \Sigma^*$ where $\vert\wordcommand{w}\vert \leq \vert\wordcommand{v}\vert$.
The word $\wordcommand{w}$ is  \emph{bounded} (resp. strictly bounded) by $\wordcommand{s} \sqsubseteq_{\vert\wordcommand{w}\vert} \wordcommand{v}$ if $\wordcommand{s} \leq \wordcommand{w}$ (resp. $\wordcommand{s} < \wordcommand{w}$) and there is no $\wordcommand{u} \sqsubseteq_{\vert\wordcommand{w}\vert} \wordcommand{v}$ such that $\wordcommand{s} < \wordcommand{u} \leq \wordcommand{w}$.
\label{def:bounding}
\end{definition}


\begin{proposition}[\cite{Adamson2021}]
Let $\wordcommand{v} \in \Sigma^n$.
The array $WX[\wordcommand{s} \sqsubseteq \wordcommand{v},x \in \Sigma]$, such that $WX[\wordcommand{s},x]$ strictly bounds $\wordcommand{w}:x$ for every $\wordcommand{w}$ strictly bounded by $\wordcommand{s}$, can be computed in $O(k \cdot n^3 \cdot \log(n))$ time where $|\Sigma| = k$.
\label{prop:complexity_WX}
\end{proposition}

\noindent
For the remainder of this paper, we can assume that the array $WX$ has been precomputed for every $\wordcommand{s} \sqsubseteq \wordcommand{v},x \in \Sigma$.
Note that in our case $k = 2$, therefore the process of computing $WX$ requires only $O(n^3 \cdot \log(n))$ time.

\section{Ranking}
\label{sec:ranking}

In this section we present our ranking algorithm.
For the remainder of this section, we assume that we are ranking the word $\word{w}$ that is the canonical representation of the binary unlabelled necklace $\necklace{w}$.
We first provide an overview of the main idea behind our ranking algorithm.

\begin{theorem}
Let $RankAN(\word{w},m)$ be the rank of the word $\word{w} \in \Sigma^n$ within the set of non-symmetric unlabelled necklaces of length $n$ that do not enclose $\word{w}$, let $RankSN(\word{w},m)$ be the rank of $\word{w}$ within the set of symmetric necklaces of length $m$ and let $RankEN(\word{w},m)$ be the rank of $\word{w}$ within the set of necklaces of length $m$ that enclose $\word{w}$.
The rank of any necklace $\necklace{w}$ represented by the word $\word{w}$ within the set of binary unlabelled necklaces of length $m$ is given by $RankAN(\word{w},m) + RankSN(\word{w},m) + RankEN(\word{w},m)$.
Further the rank can be found in $O(n^6 \log^2 n)$ time for any $m \leq n$.
\end{theorem}

\begin{proof}
Observe that every unlabelled necklace must be one of the above classes.
Therefore the rank of $\word{w}$ within the set of all binary unlabelled necklaces of length $m$ is given by $RankAN(\word{w},m) + RankSN(\word{w},m) + RankEN(\word{w},m)$.
Lemma \ref{lem:rankAN} shows that the rank of $\word{w}$ within the set of non-symmetric unlabelled necklaces of length $m$ that do not enclose $\word{w}$ can be found in $O(n^6 \log^2(n))$ time.
Theorem \ref{thm:symmetric} shows that the rank of $\word{w}$ within the set of  symmetric necklaces can be found in $O(n^6 \log^2 n)$ time.
Theorem \ref{thm:enclosing} shows that the rank of $\word{w}$ within the set of necklaces enclosing $\word{w}$ can be found in $O(n^6 \log n)$ time.
\end{proof}

\begin{lemma}
\label{lem:rankAN}
Let $RankAN(\word{w},m)$ be the rank of $\word{w}$ within the set of non-symmetric unlabelled necklaces of length $m$ that do not enclose $\word{w}$, and let $RankN(\word{w},m)$ be the rank of $\word{w}$ within the set of all necklaces of length $m$.
Then $RankAN(\word{w},m) = (RankN(\word{w},m) - RankSN(\word{w},m) - RankEN(\word{w},m))/2$.
Further, this rank can be found in $O(n^6 \log^2 n)$ time for any $m \leq n$.
\end{lemma}

\begin{proof}
Note that any asymmetric unlabelled necklace appears exactly twice in the set of necklaces smaller than $\word{w}$.
Further, any enclosing or symmetric necklace appears exactly once in the same set.
Therefore $RankAN(\word{w},m) = \frac{RankN(\word{w},m) - RankSN(\word{w},m) - RankEN(\word{w},m)}{2}$.
As the value of $RankN(\word{w},m)$ can be found in $O(n^2)$ time using the algorithm due to Sawada and Williams \cite{Sawada2017}, the value of $RankSN(\word{w},m)$ found in $O(n^6 \log^2 n)$ time from Theorem \ref{thm:symmetric}, and the of $RankEN(\word{w},m)$ found in 
$O(n^6 \log n)$ time from Theorem \ref{thm:enclosing}, the total time complexity is $O(n^6 \log^2 n)$.
\end{proof}




\section{Symmetric Necklaces}
\label{sec:symmetric}

In this section we show how to rank a word $\word{w}$ within the set of symmetric necklaces of length $m$.
Before presenting our computational tools, we first introduce the key theoretical results that form the basis for our ranking approach.
The key observation is that any symmetric necklace $\necklace{v}$ must have a period of length $2 \cdot r$ where $r$ is the smallest rotation such that $\Angle{\necklace{v}}_r = S(\Angle{\necklace{v}})$.
This is formally proven in Proposition \ref{prop:period}, and restated in Observation \ref{obs:aperiodicity} in terms of Lyndon words.

\begin{proposition}
\label{prop:period}
A necklace $\necklace{w}$ represented by the word $\word{w} \in \Sigma^n$ is symmetric if and only if there exists some $r \in [n]$ s.t. $\word{w}_i = S(\word{w}_{i + r \bmod n})$ for every $i \in [n]$.
Further, the period of $\word{w}$ equals $2 \cdot r$ where $r \in [n]$ is the smallest rotation such that $\Angle{\word{w}}_r = S(\word{w})$.
\end{proposition}

\begin{proof}
As $\necklace{w}$ is symmetric, $S(\word{w})$ must belong to the necklace class $\necklace{w}$.
Therefore, there must be some rotation $r$ such that $\Angle{\word{w}}_r = S(\word{w})$.
We now claim that $r \leq \frac{n}{2}$.
Assume for the sake of contradiction that $r > \frac{n}{2}$.
Then $\word{w}_i = S(\word{w}_{i + r \bmod{n}}) = \word{w}_{i + 2r \bmod{n}} = \hdots = \word{w}_{i + 2 \cdot k \cdot r \bmod{n}} = S(\word{w}_{i + (2\cdot k + 1)r \bmod{n}})$.
As $r > \frac{n}{2}$ this sequence must imply that either $\word{w}_i = S(\word{w}_i)$, an obvious contradiction, or that there exists some smaller value $p = GCD(n,r) \leq \frac{n}{2}$ such that $\word{w}_i = S(\word{w}_{i + p \bmod n})$.
Further, $\word{w}$ must have a period of at most $2 \cdot r$.

Assume now that $r$ is the smallest rotation such that $\Angle{\word{w}}_r = S(\word{w})$ and for the sake of contradiction further assume that the period of $\word{w}$ is $p < r$.
Then, as $\word{w}_i = \word{w}_{i + p \bmod n}$ for every $i \in [n]$, $\word{w}_{i + r \bmod n} = \word{w}_{i + r - p \bmod n}$, hence $\word{w}_i = S(\word{w}_{i + r - p \bmod n})$, contradicting the initial assumption.
The period can not be equal to the value of $r$ as by definition $\word{w}_{i} = S(\word{w}_{i + r \bmod n})$.
Assume now that the period $p$ of $\word{w}$ is between $r$ and $2 \cdot r$.
As $\word{w}_i = \word{w}_{i + c \cdot p + 2k \cdot r \bmod n}$ for every $c, k\in \mathbb{N}$ and $i \in [n]$.
Further both $r$ and $p$ must be less than $\frac{n}{2}$.
Therefore $\word{w}_i = \word{w}_{i + ((n/p) - 1)p + 2\cdot r \bmod n} = \word{i + 2 \cdot r - p \bmod n}$ and hence $\word{w}$ is periodic in $2 \cdot r - p$.
As $p > r, 2 \cdot r - p < r$, however as no such period can exist, this leads to a contradiction.
Therefore, $2 \cdot r$ is the smallest period of $\word{w}$.
\end{proof}


\begin{lemma}
\label{lem:RA_to_RB}
Let $\mathbf{RA}(\word{w},m,S,r)$ contain the set of words belonging to an symmetric necklace smaller than $\word{w}$ such that $\word{v}_i = S(\word{v}_{i + r \bmod{m}})$ for every $\word{v} \in \mathbf{RA}(\word{w},m,S,r)$.
Further let $\mathbf{RB}(\word{w},m,S,r) \subseteq \mathbf{RA}(\word{w},m,S,r)$ contain the set of words belonging to an symmetric Lyndon word smaller than $\word{w}$ such that $r$ is the smallest value for which $\word{v}_i = S(\word{v}_{i + r \bmod{m}})$ for every $\word{v} \in \mathbf{RB}(\word{w},m,S,r)$.
The size of $\mathbf{RB}(\word{w},m,S,r)$ is given by:

$$
|\mathbf{RB}(\word{w},m,S,r)| = \sum\limits_{p | r} \mu\left(\frac{m}{p}\right) |\mathbf{RA}(\word{w},m,S,p)|
$$
%
\end{lemma}

\begin{proof}
Observe that every word in $\mathbf{RA}(\word{w},m,S,r)$ must have a unique period which is a factor of $2 \cdot r$.
Therefore, the size of $\mathbf{RA}(\word{w},m,S,r)$ can be expressed as $\sum\limits_{d | r} |\mathbf{RB}(\word{w},m,S,r)|$.
Applying the M\"{o}bius inversion formula to this equation gives $|\mathbf{RB}(\word{w},m,S,r)| = \sum\limits_{p | r} \mu\left(\frac{m}{p}\right) |\mathbf{RA}(\word{w},m,S,p)|$.
\end{proof}

\begin{observation}
\label{obs:aperiodicity}
Observe that any symmetric Lyndon word $\necklace{v}$ must have length $2 \cdot r$, where $r$ is the smallest rotation such that $\Angle{\word{v}}_r = S(\Angle{\word{v}})$.
\end{observation}

\begin{lemma}
\label{lem:rank_sl_from_RB}
Let $RankSL(\word{w},2\cdot r)$ be the rank of $\word{w}$ within the set of symmetric Lyndon words of length $2 \cdot r$.
The value of $RankSL(\word{w},r)$ is given by $\frac{|\mathbf{RB}(\word{w},2 \cdot r,S,r)|}{2 \cdot r}$.
\end{lemma}

\begin{proof}
Observe that any symmetric Lyndon word has exactly $2 \cdot r$ unique translations.
Further, as any word in $\mathbf{RB}(\word{w},2 \cdot r,S,r)$ must correspond to an aperiodic word, following Observation \ref{obs:aperiodicity}, the size of $\mathbf{RB}(\word{w},2 \cdot r,S,r)$ can be used to find $RankSL(\word{w},2 \dot r)$ by dividing the cardinality of $\mathbf{RB}(\word{w},2 \cdot r,S,r)$ by $2 \cdot r$.
\end{proof}

\begin{lemma}
\label{lem:rank_sn_from_sl}
Let $RankSN(\word{w},m,r)$ be the rank of $\word{w}$ within the set of symmetric necklaces of length $m$ such that for each such necklace $\necklace{v}$, $r$ is the smallest rotation such that $\Angle{\necklace{v}}_r = S(\Angle{\necklace{v}})$.
The value of $RankSN(\word{w},m,r)$ is given by $\sum\limits_{d | 2r} RankSL(\word{w},d)$.
\end{lemma}

\begin{proof}
Following the same arguments as in Lemma \ref{lem:RA_to_RB}, observe that every necklace counted by $RankSN(\word{w},m,r)$ must have a period that is a factor of $2\cdot r$.
Therefore, the value of $RankSN(\word{w},m,r)$ is given by $\sum\limits_{d | 2r} RankSL(\word{w},d)$.
\end{proof}

\begin{lemma}
\label{lem:rank_sn_general}
Let $RankSN(\word{w},m)$ be the rank of $\word{w}$ within the set of symmetric necklaces of length $m$ and let $RankSN(\word{w},m,r)$ be the rank of $\word{w}$ within the set of symmetric necklaces of length $m$ such that for each such necklace $\necklace{v}$, $r$ is the smallest rotation such that $\Angle{\necklace{v}}_r = S(\Angle{\necklace{v}})$.
The value of $RankSN(\word{w},m)$ is given by $\sum\limits_{r | (m/2)} RankSN(\word{w},m,r)$.
\end{lemma}

\begin{proof}
Observe that every necklace counted by $RankSN(\word{w},m)$ must have a unique translation that is the minimal translation under which it is symmetric.
Further this translation must be a factor of $\frac{m}{2}$.
Therefore $RankSN(\word{w},m) = \sum\limits_{r | (m/2)} RankSN(\word{w},m,r)$.
\end{proof}

\noindent
Following Lemmas \ref{lem:RA_to_RB}, \ref{lem:rank_sl_from_RB}, \ref{lem:rank_sn_from_sl}, and \ref{lem:rank_sn_general} the main challenge in computing $RankSN(\word{w},m)$ is computing the size of $\mathbf{RA}(\word{w},m,S,r)$.
In order to do so, $\mathbf{RA}(\word{w},m,S,r)$ is partitioned into two sets, $\alpha(\word{w},r,j)$ and $\beta(\word{w},r,j)$ where $j \in [r] $.
Let $\word{v}$ be some arbitrary word in the set $\mathbf{RA}(\word{w},m,S,r)$.
The set $\alpha(\word{w},r,j)$ contains the word $\word{v}$ if $j$ is the smallest rotation under which $\Angle{\word{v}}_j \leq \word{w}$.
The set $\beta(\word{w},r,j)$ contains $\word{v}$ if $j$ is the smallest rotation under which $\Angle{\word{v}}_j \leq \word{w}$ and $\Angle{\word{v}}_t > \word{w}$ for every $t \in [r + 1, 2\cdot r]$.
Note that by this definition, $\beta(\word{w},r,j) \subseteq \alpha(\word{w},r,j)$.


\begin{observation}
\label{obs:beta_mapping}
Given any word $\word{v} \in \mathbf{RA}(\word{w},m,S,r)$ such that $\word{v} \notin \alpha(\word{w},r,j)$ for any $j \in [r]$, there exists some $j' \in [r]$ for which $\Angle{\word{v}}_{r} \in \beta(\word{w}, r, j')$.
\end{observation}

\begin{proof}
As $\word{v} \in \mathbf{RA}(\word{w},m,S,r)$, there must be some rotation $t$ such that $\Angle{\word{v}}_{t} < \word{w}$.
As $\word{v} \notin \alpha(\word{w},r,j)$, $t$ must be greater than $r$.
Therefore, $\Angle{\word{v}}_r$ must belong to $\beta(\word{w}, r, t - r)$ confirming the observation.
\end{proof}


\begin{observation}
\label{obs:alpha_mapping}
For any $\word{v} \in \beta(\word{w},r,j)$, $\Angle{\word{v}}_{r} \notin  \alpha(\word{w},r,j')$ for any $j' \in [r]$.
\end{observation}

\begin{proof}
As $\word{v} \in \beta(\word{w},r,j)$, for any rotation $t > r, \Angle{\word{v}}_t > \word{w}$.
Therefore $\Angle{\word{v}}_t \notin \alpha(\word{w},r,j')$ for any $j' \in [r]$.
\end{proof}


\noindent
Combining Observations \ref{obs:beta_mapping} and \ref{obs:alpha_mapping}, the size of $\mathbf{RA}(\word{w},m,S,r)$ can be given in terms of the sets $\alpha(\word{w},r,j)$ and $\beta(\word{w},r,j)$ as $\sum\limits_{j \in [r]} |\alpha(\word{w},r,j)| + |\beta(\word{w},r,j)|$.
%
%
%
%
%
%
%
%
%
The remainder of this section is laid out as follows.
We first provide a high level overview of how to compute the size of $\alpha(\word{w},r,j)$.
Then we provide a high level overview on computing the size of $\beta(\word{w},r,j)$.
Finally, we state Theorem \ref{thm:symmetric}, summarising the main contribution of this section and showing that $RankSN(\word{w},m)$ can be computed in at most $O(n^6 \log^2 n)$ time.

\subsection{Computing the size of $\alpha(\word{w},r,j)$.}

We begin with a formal definition of $\alpha(\word{w},r,j)$.
Let $\alpha(\word{w},r,j) \subseteq \mathbf{RA}(\word{w},m,S,r)$ be the subset of words in $\mathbf{RA}(\word{w},m,S,r)$ such that for every word $\word{v} \in \alpha(\word{w},r,j)$, $j$ is the smallest rotation for which $\Angle{\word{v}}_j \leq \word{w}$.
Note that if $j$ is the smallest rotation such that $\Angle{\word{v}}_j \leq \word{w}$, the first $j$ symbols of $\word{v}$ must be such that for every $j' \in [j - 1], \word{v}_{[j', 2r]} > \word{w}$.
Let $\mathbf{A}(\word{w}, p, \word{B}, i, j, r) \subseteq \alpha(\word{w},r,j)$ be the set of words of length $2 \cdot r$ such that every word $\word{v} \in \mathbf{A}(\word{w}, p, \word{B},i, j,r)$:
\begin{enumerate}
    \item \label{cond:greater_than_before_j} $\Angle{\word{v}}_{s} > \word{w}$ for every $s \in [j - 1]$.
    \item \label{cond:less_than_at_j} $\Angle{\word{v}}_j < \word{w}$.
    \item \label{cond:symmetric} $\word{v}_{[1,r]} = S(\word{v}_{[r + 1, 2 \cdot r]})$.
    \item \label{cond:back_bound} The subword $\word{v}_{[r + 1, r + i]}$ is strictly bound by $\word{B} \sqsubseteq_i \word{w}$.
    \item \label{cond:front_prefix} The subword $\word{v}_{[i - p, i]} = \word{w}_{[1,p]}$.
\end{enumerate}

\noindent
Rather than computing the size of $\mathbf{A}(\word{w}, p, \word{B}, i, j, r)$ directly, we are instead interested in the number of unique suffixes of length $r - i$ of the words in $\mathbf{A}(\word{w}, p, \word{B}, i, j, r)$.
Note that as every word in $\mathbf{A}(\word{w}, p, \word{B}, i, j, r)$ belongs to a symmetric necklace, the number of possible suffixes on length $r - i$ of words in $\mathbf{A}(\word{w}, p, \word{B}, i, j, r)$ equals the number of unique subwords of words in $\mathbf{A}(\word{w}, p, \word{B}, i, j, r)$ between position $i + 1$ and $r$.
Let $SA(\word{w}, p, \word{B}, i, j, r)$ be a function returning the number of unique suffixes of length $r - i$ of the words within $\mathbf{A}(\word{w}, p, \word{B}, i, j, r)$.
The value of $SA(\word{w}, p, \word{B}, i, j, r)$ is computed in a dynamic manner relaying on a key structural proposition regarding $\mathbf{A}(\mathbf{A}(\word{w}, p, \word{B}, i, j, r))$.


\begin{proposition}
\label{prop:alpha_structure}
Given $\word{v} \in \mathbf{A}(\word{w}, p, \word{B}, i, j, r)$, such that $\word{v}_{[i - s, i + 1]} \geq \word{w}_{[1,s]}$ for every $s \in [i]$,  $\word{v}$ also belongs to $\mathbf{A}( \word{w},p',WX[\word{B},\word{v}_{i + 1}],i + 1, j, r)$ where $p' = p + 1$ if $\word{v}_{i + 1} = \word{w}_{p + 1}$ and $0$ otherwise.
\end{proposition}

\begin{proof}
By definition, if $\word{v} \in \mathbf{A}(\word{w}, p, \word{B}, i, j, r)$ then there must exists some $p' \in [i + 1]$, and $\word{B}\sqsubseteq_i \word{w}$ such that $\word{v} \in \mathbf{A}(\word{w}, p', \word{B}', i, j, r)$.
From Proposition \ref{prop:complexity_WX}, the value of $\word{B}' = WX[\word{B},S(\word{v}_{i + 1})]$.
Further $\word{v}_{i + 1} \geq \word{w}_{p + 1}$ as otherwise $\word{v}_{[i - p, i + 1]} < \word{w}_{[1,p + 1]}$, contradicting the original assumption.
If $\word{v}_{i + 1} = \word{w}_{p + 1}$ then $p' = p + 1$ by definition.
Otherwise $p' = 0$ as $\word{v}_{[i - s,i + 1]} > \word{w}_{[1,s + 1]}$.
\end{proof}

\begin{corollary}
\label{col:alpha_cartesian}
Let $\word{v},\word{u} \in \mathbf{A}(\word{w}, p, \word{B}, i, j, r)$ be a pair of words and let $\word{v}' = \word{u}_{[1,i]} : \word{v}_{[i + 1, r]} : S(\word{v}_{[1,i]} : \word{u}_{[i + 1, r]})$.
Then $\word{v}' \in \mathbf{A}( \word{w},p',WX[\word{B},\word{v}_{i + 1}],i + 1, j, r)$ if and only if $\word{v} \in \mathbf{A}( \word{w},p',WX[\word{B},\word{v}_{i + 1}],i + 1, j, r)$.
\end{corollary}

\noindent
Proposition \ref{prop:alpha_structure} and Corollary \ref{col:alpha_cartesian} provide the basis for computing the value of $SA(\word{w}, p, \word{B}, i, j, r)$.
This is done by considering 4 cases based on the value of $i$ relative to the values of $j$ and $r$ which we will sketch bellow.
The key observation behind this partition is that the value of the symbol at position $i + 1$ is restricted differently depending on the values of $i,j,r$ and $p$.

\begin{lemma}
\label{lem:i_le_j}
For any $i < j$, the value of $SA(\word{w}, p, \word{B}, i, j,r)$ equals:
\begin{itemize}
    \item $SA(\word{w},p + 1, WX[\word{B}, 0], i + 1,j,r)|$ if $\word{w}_{p + 1} = 1$.
    \item $SA(\word{w},p + 1, WX[\word{B}, 1], i + 1,j,r)| + SA(\word{w},0, WX[\word{B}, 0], i + 1,j,r)$ if $\word{w}_{p + 1} = 0$.
\end{itemize}
\end{lemma}

\begin{proof}
We prove this lemma in a piece-wise manner.
Note that given the first $i$ symbols of any word $\word{v} \in \mathbf{A}(\word{w}, p, \word{B}, i, j,r)$, there are two possibilities for the $i + 1^{th}$ symbol.
Either $\word{v}_{i + 1} = 0$ or $\word{v}_{i + 1} = 1$.
In the first case, when $\word{w}_{p + 1} = 1$ then to satisfy Condition \ref{cond:greater_than_before_j}, $\word{v}_{i + 1}$ must equal $1$, as otherwise $\word{v}_{[i - p, i + 1]} < \word{w}_{[1,p + 1]}$.
As $\word{v}_{[i - p,p + 1]} = \word{w}_{[1,p + 1]}$, and $S(\word{v}_{[i - p,p + 1]})$ is bounded by $WX[\word{B},0]$, the value of $SA(\word{w}, p, \word{B}, i, j,r)$ in this case equals the value of $SA(\word{w},p + 1, WX[\word{B}, 0], i + 1,j,r)$.

If $\word{w}_{[p + 1]} = 0$ then the value of $\word{v}_{i + 1}$ can be either $0$ or $1$.
The number of suffixes of length $r - i$ of words in $\mathbf{A}(\word{w}, p, \word{B}, i, j,r)$ where the symbol at position $i + 1$ is 0 is equal to the value of $SA(\word{w}, p + 1, WX[\word{B},1], i + 1, j,r)$.
Similarly, the number of suffixes of length $r - i$ of words in $\mathbf{A}(\word{w}, p, \word{B}, i, j,r)$ where the symbol at position $i + 1$ is 1 is equal to the value of $SA(\word{w}, p, \word{B}, i, j,r)$ for which the symbol at position $i + 1$ equals $1$ is $|\mathbf{A}(\word{w}, p + 1, WX[\word{B},0], i + 1, j,r)|$.
\end{proof}


\begin{lemma}
\label{lem:i_e_j}
Let $i = j$.
The value of $SA(\word{w}, p, \word{B}, i, j,r)$ equals:
\begin{itemize}
    \item 0 if $p > 0$.
    \item $SA(\word{w}, 1, WX[\word{B},1], i + 1, j,r)$ if $p = 0$.
\end{itemize}
\end{lemma}

\begin{proof}
Consider some word $\word{v} \in \mathbf{A}(\word{w}, p, \word{B}, i, j,r)$.
In the first case, as $\word{v}_{[j,2 \cdot r]} : \word{v}_{[1,j-1]} \leq \word{w}$, then $\word{w}_{[1,p_f]} : \word{v}_{[j,2 \cdot r]} : \word{v}_{[1,j-1]} $ must also be smaller than $\word{w}$, contradicting the assumption that $\word{v}$ is in $\alpha(\word{w},j,r)$.
Therefore, $p$ must equal $0$.
In the second case, for $\Angle{\word{v}}_{j}$ to be less than $\word{w}$, $\word{v}_{i + 1}$ must equal $0$.
Hence the value of $SA(\word{w}, p, \word{B}, i, j,r)$ is equal to the value of $SA(\word{w}, 1, WX[\word{B},1], i + 1, j,r)$.
\end{proof}

\begin{lemma}
\label{lem:i_ge_j}
Let $i > j$ and further let $p = i - j$.
The value of $SA(\word{w}, p, \word{B}, i, j,r)$ equals:
\begin{itemize}
    \item $0$ if $i = r$ and $\word{w}_{[1,p]} : \word{B} \geq \word{w}_{[1,p + 1]}$.
    \item $1$ if $i = r$ and $\word{w}_{[1,p]} : \word{B} < \word{w}_{[1,p + 1]}$.
    \item $SA(\word{w}, p + 1, WX[\word{B},1], i + 1, j,r)$ if $i < r$ and $\word{w}_{p + 1} = 0$.
    \item $SA(\word{w}, p + 1, WX[\word{B},0], i + 1, j,r) + 2^{r - i - 1}$ if $\word{w}_{p + 1} = 1$.
\end{itemize}
\end{lemma}

\begin{proof}
Consider some word $\word{v}\in \mathbf{A}(\word{w}, p, \word{B}, i, j,r)$
In the first case, as $\word{v}_{[r + 1, 2 \cdot r]} > \word{B}$, if $\word{w}_{[1,p]} : \word{B} \geq \word{w}_{[1,p_f + 1]}$, then $\Angle{\word{v}}_j > \word{w}$, contradicting Condition \ref{cond:less_than_at_j}.
In the second case, note that as $\word{B} < \word{w}_{[p + 1, r + p]}$, then by the definition of strictly bounding subwords, $\word{v}_{[r + 1,2 \cdot r]} < \word{w}_{[p + 1, r + p]}$.
Further, as $i = r$ the only possible 0 length suffix of words in $\mathbf{A}(\word{w}, p, \word{B}, i, j,r)$ is the empty word.
Therefore the value of $SA(\word{w}, p, \word{B}, i, j,r)$ is 1.

In the third case, following the same arguments as in Lemma \ref{lem:i_le_j}, the only possible value for $\word{v}_{i + 1}$ is $0$, as a value of $1$ would contradict Condition \ref{cond:less_than_at_j}.
Further, as $\word{v}_{[i - p, i + 1]} = \word{w}_{[1,p + 1]}$, the value of $SA(\word{w}, p, \word{B}, i, j,r)$ equals $SA(\word{w}, p + 1, WX[\word{B},0], i + 1, j,r)$.
In the fourth case, the value of $\word{v}_{i + 1}$ can be either $0$ or $1$.
If $\word{v}_{i + 1} = 0$ then as $\word{v}_{[i - p, i + 1]} < \word{w}_{[1,p + 1]}$, the remaining symbols in $\word{v}$ can be any arbitrary choice.
As there are $2^{r - i - 1}$ such possible suffixes, the number of unique suffixes of length $r - i$ of words in $\mathbf{A}(\word{w}, p, \word{B}, i, j,r)$ where the symbol at position $i + 1$ is $0$ is $2^{r - i - 1}$.
Following the arguments outlined above, the value of $SA(\word{w}, p, \word{B}, i, j,r)$ where $\word{v}_{i + 1} = 1$ is equal to $SA(\word{w}, p + 1, WX[\word{B},0], i + 1, j,r)$.
\end{proof}

\begin{lemma}
\label{lem:size_of_A}
The value of $SA(\word{w},p,\word{B}, i,j,r)$ can be computed in $O(n^3)$ time.
\end{lemma}


\begin{proof}
We use the dynamic programming approach explicitly given in Lemma \ref{lem:i_le_j}, \ref{lem:i_e_j}, and \ref{lem:i_ge_j}.
Starting with the case where $i = r$, the value of $SA(\word{w},p,\word{B}, i,j,r)$ can be computed in $O(n)$ time by comparing $\word{w}_{[1,p]} : \word{B}$ with $\word{w}_{[1,p + 1]}$.
In the case where $i > j$ while still being smaller than $r$, the value of $SA(\word{w},p,\word{B}, i,j,r)$ can be computed in $O(1)$ time if the value of $SA(\word{w},p',\word{B}', i + 1,j,r)$ has been precomputed for every $p'\in [i], \word{B}' \sqsubseteq_i \word{w}$ using Lemma \ref{lem:i_ge_j}.
Similarly, when $i = j$ the value of $SA(\word{w},p,\word{B}, i,j,r)$ can be computed in $O(1)$ time using Lemma \ref{lem:i_e_j} provided that the value of $SA(\word{w},p',\word{B}', i + 1,j,r)$. has been precomputed for every $p'\in [i], \word{B}' \sqsubseteq_i \word{w}$.
Finally when $i < j$, the value of $SA(\word{w},p,\word{B}, i,j,r)$ can be computed in $O(1)$ time using Lemma \ref{lem:i_le_j} provided that the size of $SA(\word{w},p',\word{B}', i + 1,j,r)$. has been precomputed for every $p'\in [i], \word{B}' \sqsubseteq_i \word{w}$.

This leads to a natural dynamic programming formulation.
Starting with $i = j$, the value of $SA(\word{w},p,\word{B}, i,j,r)$ can be computed in $O(n^3)$ time for every $p \in [i]$ and $\word{B} \sqsubseteq_i \word{w}$.
Once this has been computed, the value of $SA(\word{w},p,\word{B}, i,j,r)$ can be computed in decreasing value of $i$ for every $i,p \in [r], \word{B} \sqsubseteq_i \word{w}$ in $O(1)$ time per value, for a total of $O(n^3)$ time complexity. 
\end{proof}

\begin{lemma}
\label{lem:counting_alpha}
The size of $\alpha(\word{w},j,r)$ can be computed in $O(n^4)$ time.
\end{lemma}

\begin{proof}
From Lemma \ref{lem:size_of_A}, the value of $SA(\word{w},p,\word{B}, i,j,r)$ can be computed in $O(n^3)$ time for any value of $i,p \in [n]$ and $\word{B} \sqsubseteq_i \word{w}$.
Note that $SA(\word{w},0,\emptyset, 0,j,r)$ allows us to count the number of words $\word{v} \in \alpha(\word{w},j,r)$ where $\word{v}_{[r + 1, r + i]} \not\sqsubseteq \word{w}$ for every $i \in [r]$, or equivalently, where $S(\word{v}_{[1,i]}) \not\sqsubseteq \word{w}$.
To compute the remaining words, let $\word{u} \sqsubseteq_{i - 1} \word{w}$ and let $x \in \{0,1\}$ be a symbol such that $\word{u} : x \not\sqsubseteq \word{w}$.
Further let $\word{B} \sqsubseteq_{i} \word{w}$ be the subword of $\word{w}$ strictly bounding $\word{u} : x$ and let $p$ be the length of the longest suffix of $S(\word{u} : x)$ that is a prefix of $\word{w}$, i.e. the largest value such that $S(\word{u} : x)_{[i - p: i]} = \word{w}_{[1,p]}$.
Observe that $S(\word{u} : x)_{[1,p]}$ is the prefix of some word $\word{v} \in \alpha(\word{w},j,r)$ if and only if one of the following holds:
\begin{itemize}
    \item if $i < r$ then $(\word{u} : x)_{[i - s, i]} > \word{w}_{[1,s]}$ for every $s \in [p + 1,i]$.
    \item if $i = r$ then $p = 0$.
    \item if $i > r$ then $p = i - r$.
\end{itemize}

\noindent
As each condition can be checked in at most $O(n)$ time, and there are at most $O(n^2)$ subwords of $\word{w}$, it is possible to check for every such subword if it is a prefix of some word in $\mathbf{A}(\word{w},p,\word{B}, i,j,r)$ in $O(n^3)$ time.
Following Corollary \ref{col:alpha_cartesian}, the number of suffixes of each word in $\mathbf{A}(\word{w},p,\word{B}, i,j,r)$ is equal to the value of $SA(\word{w},p,\word{B}, i,j,r)$.
By precomputing $SA(\word{w},p,\word{B}, i,j,r)$, the number of words in $\alpha(\word{w},j,r)$ with $\word{u} : x$ as a prefix can be computed in $O(1)$ time.
Therefore the total complexity of computing the size of $\alpha(\word{w},j,r)$ is $O(n^3)$.
\end{proof}

\subsection{Computing the size of $\beta(\word{w},r,j)$.}

We start by subdividing $\beta(\word{w},r,j)$ into the subsets $\mathbf{B}(\word{w}, p_f, p_b, \word{B_f}, \word{B_b}, i, j,r)$.
Let $\mathbf{B}(\word{w}, p_f, p_b, \word{B_f}, \word{B_b}, i, j,r) \subseteq \beta(\word{w},r,j)$ be the subset of $\beta(\word{w},r,j)$ containing every word $\word{v} \in \beta(\word{w},r,j)$ where $\word{v}$ satisfies:
\begin{enumerate}
    \item $\word{v}_{[1,r]} = S(\word{v}_{[r + 1, 2 \cdot r]})$.
    \item The first $i$ symbols of $\word{v}$ are strictly bound by $\word{B_f} \sqsubseteq_i \word{w}$ ($\word{B_f}$ standing for \textbf{b}ounding the \textbf{f}ront).
    \item The subword $\word{v}_{[r + 1, r + i]}$ is strictly bound by $\word{B_b} \sqsubseteq_i \word{w}$ ($\word{B_b}$ standing for \textbf{b}ounding the \textbf{b}ack).
    \item The subword $\word{v}_{[i - p_f, i]} = \word{w}_{[p_f]}$ ($p_f$ standing for the \textbf{f}ront \textbf{p}refix).
    \item the subword $\word{v}_{[r + i - p_b, r + i]} = \word{w}_{[p_b]}$ ($p_b$ standing for the \textbf{b}ack \textbf{p}refix).
\end{enumerate}

\begin{proposition}
\label{prop:beta_structure}
Given $\word{v} \in \mathbf{B}(\word{w},p_f,p_b,\word{B_f},\word{B_b},i,j,r)$, where $\word{v}_{[i - s, i + 1]} \geq \word{w}_{[1,s]}$ for every $s \in [i]$, $\word{v}$ also belongs to $\mathbf{B}(\word{w}, p_f', p_b', XW[\word{B_f},\word{v}_{i + 1}], XW[\word{B_b},S(\word{v}_{i + 1})],i,j,r)$ where $p_f' = p_f + 1$ if $\word{v}_{i + 1} = w_{p_f + 1}$ or 0 otherwise, and $p_b' = p_b + 1$ if $S(\word{v}_{i + 1}) = \word{w}_{p_b + 1}$, and 0 otherwise.
\end{proposition}

\begin{proof}
Following the same arguments as Proposition \ref{prop:alpha_structure}, observe that $\word{v}_{[1,i + 1]}$ is bound by $XW[\word{B_f},\word{v}_{i + 1}]$ and $S(\word{v}_{[1,i + 1]})$ is bound by $XW[\word{B_b},S(\word{v}_{i + 1})]$.
Similarly, the value of $p_f'$ is $p_f + 1$ if and only if $\word{v}_{i + 1} = \word{w}_{p_f + 1}$, and must be 0 otherwise.
Further the value of $p_b'$ is $p_b + 1$ if and only if $S(\word{v}_{i + 1}) = \word{w}_{p_b + 1}$, and 0 otherwise.
\end{proof}

\begin{corollary}
\label{col:beta_cartesian}
Let $\word{v},\word{u} \in \mathbf{B}(\word{w},p_f,p_b,\word{B_f},\word{B_b},i,j,r)$ be a pair of words and let $\word{v}' = \word{u}_{[1,i]} : \word{v}_{[i + 1, r]} : S(\word{v}_{[1,i]} : \word{u}_{[i + 1, r]})$.
Then $\word{v}' \in \mathbf{B}(\word{w}, p_f', p_b', \word{B_f}', \word{B_b}',i,j,r)$ if and only if $\word{v} \in \mathbf{B}(\word{w}, p_f', p_b', \word{B_f}', \word{B_b}',i,j,r)$.
\end{corollary}

\noindent
Proposition \ref{prop:beta_structure} and Corollary \ref{col:beta_cartesian} are used in an analogous manner the Proposition \ref{prop:alpha_structure}.
As before, the goal is not to directly compute the size of $\mathbf{B}(\word{w},p_f,p_b,\word{B_f},\word{B_b},i,j,r)$, but rather to compute the number of suffixes of length $r - i$ of the words therein.
To that end, let $SB(\word{w},p_f,p_b,\word{B_f},\word{B_b},i,j,r)$ be the number of unique suffixes of length $r - i$ of words in $\mathbf{B}(\word{w},p_f,p_b,\word{B_f},\word{B_b},i,j,r)$.
Note that the number of suffixes of length $r - i$ of words in $\mathbf{B}(\word{w},p_f,p_b,\word{B_f},\word{B_b},i,j,r)$ equals the number of unique subwords between positions $i + 1$ and $r$ of the words $\mathbf{B}(\word{w},p_f,p_b,\word{B_f},\word{B_b},i,j,r)$.
Additionally, note that following Corollary \ref{col:beta_cartesian}, the size of $\mathbf{B}(\word{w},p_f,p_b,\word{B_f},\word{B_b},i,j,r)$ can be computed by taking the product of the number of unique prefixes of words in $\mathbf{B}(\word{w},p_f,p_b,\word{B_f},\word{B_b},i,j,r)$, and the number of unique suffixes of words in $\mathbf{B}(\word{w},p_f,p_b,\word{B_f},\word{B_b},i,j,r)$.
The process of computing the number of such suffixes is divided into four cases based on the values of $i,j$ and $r$.


\begin{lemma}
\label{lem:i_le_j_beta}
For any $i < j$, the value of $SB(\word{w}, p_f, p_b, \word{B_f}, \word{B_b}, i, j,r)$ equals:
\begin{itemize}
    \item 0 if $\word{w}_{p_f + 1} = 1 \text{ and } \word{w}_{p_b + 1} = 1$.
    \item $SB(\word{w}, p_f + 1, p_b + 1, WX[\word{B_f},0], WX[\word{B_b},1], i + 1, j,r)$ if $\word{w}_{p_f} = 0 \text{ and } \word{w}_{p_b} = 1$.
    \item $SB(\word{w}, p_f + 1, p_b + 1, WX[\word{B_f},1], WX[\word{B_b},0], i + 1, j,r)$ if $\word{w}_{p_f} = 1 \text{ and } \word{w}_{p_b} = 0$.
    \item $SB(\word{w}, p_f + 1, 0, WX[\word{B_f},0], WX[\word{B_b},1], i + 1, j,r)$ + \\$SB(\word{w}, 0, p_b + 1, WX[\word{B_f},1], WX[\word{B_b},0], i + 1, j,r)$ if $\word{w}_{p_f} = 0 \text{ and } \word{w}_{p_b} = 0$
\end{itemize}
\end{lemma}

\begin{proof}
We prove this lemma in a piece wise manner.
Note that given the first $i$ symbols of any word $\word{v} \in \mathbf{B}(\word{w}, p_f, p_b, \word{B_f}, \word{B_b}, i, j,r)$, there are two possibilities for the $i + 1^{th}$ symbol.
Either $\word{v}_{i + 1} = 0$ or $\word{v}_{i + 1} = 1$.
In the first case, when $\word{w}_{p_f + 1} = 1 \text{ and } \word{w}_{p_b + 1} = 1$, if $\word{v}_{i + 1} = 0$ then $\word{v}_{[i - p_f, i + 1]} < \word{w}$, contradicting the assumption that $\word{v} \in \beta(\word{w},j,r)$.
Similarly, if $\word{v}_{i + 1} = 1$ then $\word{v}_{r + i + 1} = 0$, and therefore $\word{v}_{[r + i - p_b, r + i + 1]} < \word{w}$.
Therefore, in this case there can be no possible words in $\mathbf{B}(\word{w}, p_f, p_b, \word{B_f}, \word{B_b}, i, j,r)$.

In the second case, where $\word{w}_{p_f} = 0 \text{ and } \word{w}_{p_b} = 1$, if $\word{v}_{i + 1} = 1$ then following the same arguments as above $\word{v}_{[r + i - p_b, r + i + 1]} < \word{w}$ leading to a contradiction.
If $\word{v}_{i + 1} = 0$ then $\word{v}_{[i - p_f, i + 1]} = \word{w}_{[1, p_f + 1]}$ and $\word{v}_{[r + i - p_b, r + i + 1]} = \word{w}_{[1, p_b + 1]}$, therefore the value of $SB(\word{w}, p_f, p_b, \word{B_f}, \word{B_b}, i, j,r)$ exactly equals $SB(\word{w}, p_f + 1, p_b + 1, WX[\word{B_f},0], WX[\word{B_b},1], i + 1, j,r)$.

In the third case where $\word{w}_{p_f} = 1 \text{ and } \word{w}_{p_b} = 0$, if $\word{v}_{i + 1} = 0$ then following the same arguments as in the first case $\word{v}_{[i - p_f,i + 1]} < \word{w}$ leading to a contradiction.
Therefore, $\word{v}_{i + 1} = 1$.
And hence $\word{v}_{[i - p_f, i + 1]} = \word{w}_{[1, p_f + 1]}$ and $\word{v}_{[r + i - p_b, r + i + 1]} = \word{w}_{[1, p_b + 1]}$, therefore the value of $SB(\word{w}, p_f, p_b, \word{B_f}, \word{B_b}, i, j,r)$ exactly equals $SB(\word{w}, p_f + 1, p_b + 1, WX[\word{B_f},1], WX[\word{B_b},0], i + 1, j,r)$.

In the final case, if $\word{v}_{i + 1} = 0$ then $\word{v}_{[i - p_f, i + 1]} = \word{w}_{[1, p_f + 1]}$ and $\word{v}_{[r + i - p_b, r + i + 1]} > \word{w}_{[1, p_b + 1]}$.
Hence the number of unique suffixes of length $r - i$ of words in $\mathbf{B}(\word{w}, p_f, p_b, \word{B_f}, \word{B_b}, i, j,r)$ where the symbol at position $i + 1$ is $0$ is $SB(\word{w}, p_f + 1, 0, WX[\word{B_f},0], WX[\word{B_b},1], i + 1, j,r)$.
Similarly, if $\word{v}_{i + 1} = 1$ then $\word{v}_{[i - p_f, i + 1]} > \word{w}_{[1, p_f + 1]}$ and $\word{v}_{[r + i - p_b, r + i + 1]} = \word{w}_{[1, p_b + 1]}$.
Hence the number of unique suffixes of length $r - i$ of words in $\mathbf{B}(\word{w}, p_f, p_b, \word{B_f}, \word{B_b}, i, j,r)$ where the symbol at position $i + 1$ is $1$ is $SB(\word{w}, 0, p_b + 1, WX[\word{B_f},1], WX[\word{B_b},0], i + 1, j,r)$.
Therefore the value of $SB(\word{w}, p_f, p_b, \word{B_f}, \word{B_b}, i, j,r)$ is $SB(\word{w}, p_f + 1, 0, WX[\word{B_f},0], WX[\word{B_b},1], i + 1, j,r) + SB(\word{w}, 0, p_b + 1, WX[\word{B_f},1], WX[\word{B_b},0], i + 1, j,r)$.
\end{proof}


\begin{lemma}
\label{lem:i_e_j_beta}
Let $k$ be the length of the longest prefix of $\word{w}$ such that $\word{w}_{[1,k]} = 0^k$.
The value of\\ $SB(\word{w}, p_f, p_b, \word{B_f}, \word{B_b}, j, j,r)$ is:
\begin{itemize}
    \item 0 if $p_f > 0$
    \item 0 if $k \geq r - j$ and $O^{r - j} : \word{B_b} \geq \word{w}_{[1,r]}$.
    \item 0 if $k \geq r - j$ and $\word{w}_{[1,p_b]} : 1^{r - j} : \word{B_f} < \word{w}_{[1,p_b + r]}$.
    \item 1 if $k \geq r - j$ and $O^{r - j} : \word{B_b} < \word{w}$ and none of the above conditions hold.
    \item $SB(\word{w}, k, p_b + k, WX[\word{B_f},0^k], WX[\word{B_b},1^k], j + k, j,r)$ if $\word{w}_{[p_f + 1, p_f + k]} = 1^k$ and $k < r - j$.
    \item $SB(\word{w}, k, 0,  WX[\word{B_f},0^k], WX[\word{B_b},1^k], j + k, j,r)$ if $k < r - j$ and $\word{w}_{[p_f + 1, p_f + k]} \neq 1^k$.
\end{itemize}
\end{lemma}

\begin{proof}
Observe that in order for $\word{v}_{[j,2 \cdot r]} : \word{v}_{[1,j-1]}$ to be smaller than $\word{w}$, the first $k$ symbols of $\word{v}_{[j,2 \cdot r]} : \word{v}_{[1,j-1]}$ must equal $0$.
In the first case, as $\word{v}_{[j,2 \cdot r]} : \word{v}_{[1,j-1]} \leq \word{w}$, then $\word{w}_{[1,p_f]} : \word{v}_{[j,2 \cdot r]} : \word{v}_{[1,j-1]} $ must also be smaller than $\word{w}$, contradicting the assumption that $\word{v}$ is in $\beta(\word{w},j,r)$.
In the second case, as $\word{v}_{[r,r + j]} > \word{B_b}$, if $O^{r - j} : \word{B_b} \geq \word{w}_{[1,r]}$ then $0^{r - j} : \word{v}_{[r,r + j]}$ must be strictly greater than $\word{w}_{[1,r]}$.
Similarly, in the third case, as $\word{w}_{[1,p_b]} : 1^{r - j} : \word{B_f} < \word{w}_{[1,p_b + r]}$, by the definition of strictly bounding subwords, $\word{v}_{[1,j]} < \word{w}_{[p_b + r - j,p_b + r]}$.
In the fourth case, as every rotation of the last $r$ symbols of $\word{v}$ is greater than $\word{w}$, rotation by $j$ is the smallest such rotation for which $\word{v}$ is less than $\word{w}$, and there is only a single possible value of the last $r - j$ symbols, there is only a single possible suffix of length $r - i$ of the words in $\mathbf{B}(\word{w}, p_f, p_b, \word{B_f}, \word{B_b}, j, j,r)$.

In the fifth case, if the subword $\word{w}_{[p_b + 1, p_b + k]}$ equals $1^k$ then the value of $SB(\word{w}, p_f, p_b, \word{B_f}, \word{B_b}, j, j,r)$ equals the value of $SB(\word{w}, k, p_b + k, WX[\word{B_f},0^k], WX[\word{B_b},1^k], j + k, j,r)$.
Finally, if the above condition does not hold, as the first symbol in $\word{w}$ must be 0, the longest prefix of $\word{w}$ that can be made from a suffix of $\word{v}_{[r,r + j +k]}$ is the empty word.
Hence the value of $SB(\word{w}, p_f, p_b, \word{B_f}, \word{B_b}, j, j,r)$ is $SB(\word{w}, k, 0,  WX[\word{B_f},0^k], WX[\word{B_b},1^k], j + k, j,r)$.
\end{proof}

\noindent
Before we present the final case, a further auxiliary set $\mathbf{Y}(\word{w},i,p_b, \word{B_f})$ is introduced containing the set of words of length $i$ where every word $\word{v} \in \mathbf{Y}(\word{w}, i,p_b, \word{B_f})$ and $j \in [p_b + i + |\word{B_f}|]$ $(\word{w}_{[1,p_b]} : \word{v} : \word{B_f})_{[p_b + i + |\word{B_f}| - j, p_b + i + |\word{B_f}|]} > \word{w}$.
This set is used to help finish computing the value of $SB(\word{w}, p_f, p_b, \word{B_f}, \word{B_b}, j, j,r)$.
Leaving the means of computing the size of $\mathbf{Y}(\word{w}, i,p_b, \word{B_f})$ as a black box for now, Lemma \ref{lem:i_ge_j_beta} completes the method of computing  $SB(\word{w}, p_f, p_b, \word{B_f}, \word{B_b}, j, j,r)$.

\begin{lemma}
\label{lem:i_ge_j_beta}
Let $j < i$.
The value of $SB(\word{w}, p_f, p_b, \word{B_f}, \word{B_b}, i, j,r)$ is:
\begin{itemize}
    \item $0$ if $i = r$ and $\word{w}_{[1,p_f]} : \word{B_f} \geq \word{w}_{[1,p_f + r]}$.
    \item $1$ if $i = r$ and $\word{w}_{[1,p_f]} : \word{B_f} < \word{w}_{[1,p_f + r]}$.
    \item $|\mathbf{Y}(\word{w}, r - i - 1, p_b + 1, WX[\word{B_f}, 0])|$ if $\word{w}_{p_f + 1} = 1$ and $\word{w}_{p_b + 1} = 1$.
    \item $SB(\word{w}, p_f + 1, p_b + 1, WX[\word{B_f},0], WX[\word{B_b},1], i + 1, j,r)$ if $\word{w}_{p_f + 1} = 0$ and $\word{w}_{p_b + 1} = 1$.
    \item $|\mathbf{Y}(\word{w}, r - i - 1, p_b + 1, WX[\word{B_f}, 0])| + SB(\word{w}, p_f + 1, p_b + 1, WX[\word{B_f},1], WX[\word{B_b},0], i + 1, j,r)$ if $\word{w}_{p_f + 1} = 1$ and $\word{w}_{p_b + 1} = 0$.
    \item $SB(\word{w}, p_f + 1, 0, WX[\word{B_f},0], WX[\word{B_b},1], i + 1, j,r)$ if $\word{w}_{p_f + 1} = 0$ and $\word{w}_{p_b + 1} = 0$.
\end{itemize}
\end{lemma}

\begin{proof}
Note that in the first 2 cases, we assume that the recursive method has been used to determine the size up to this point and therefore, $p_f = j$.
In the first case, if $i = r$ then every symbol in $\word{v}$ is fixed.
Further, as $\word{v}_{[r - p_f, 2 \cdot r]} > \word{w}_{[1,p_f]} : \word{B_f}$, then if $\word{w}_{[1,p_f]} : \word{B_f} \geq \word{w}_{[1,p_f + r]}$, $\word{v}_{[r - p_f, 2 \cdot r]} > \word{w}_{[1,p_f + r]}$, contradicting the assumption that $\word{v} \in \beta(\word{w}, j, r)$.
Similarly, in the second case, as $\word{v}_{[r + 1, 2 \cdot 2]}$ must be less than $\word{w}_{[p_f + 1, p_f + r]}$, if $\word{w}_{[1,p_f]} : \word{B_f} < \word{w}_{[1,p_f + r]}$ then $\word{v}_{[1,p_f + r]} < \word{w}_{[1,p_f + r]}$.

In the third case, following the arguments in Lemma \ref{lem:i_le_j_beta}, the value of $\word{v}_{i + 1}$ must be 0.
Further, as $\word{w}_{p_f + 1} = 1$, $\Angle{\word{v}}_j < \word{w}$ regardless of the value of any symbol after $i$.
Therefore it is sufficient to ensure that the remaining symbols satisfy the condition that $\Angle{\word{v}}_t < \word{w}$ for ever $t \in [r + 1, 2 \cdot r]$.
As any word $\word{u} \in \mathbf{Y}(\word{w}, r - i - 1, p_b + 1, WX[\word{B_f}, 0])$ satisfies the condition that $(\word{w}_{[1,p_b]} : \word{u} : \word{B_f})_{[p_b + i - 1 + |\word{B_f}| - s, p_b + i  - 1 + |\word{B_f}|]} > \word{w}$ for every $s  \in [p_b + i + |\word{B_f}|]$, the value of $\Angle{\word{v}}_t > \word{w}$ for ever $t \in [r + 1, 2 \cdot r]$.
Further, any word not in $\mathbf{Y}(\word{w}, r - i - 1, p_b + 1, WX[\word{B_f}, 0])$ would allow for some rotation $t \in [r + 1, 2 \cdot r]$ for which $\Angle{\word{v}}_t < \word{w}$.

In the fourth case, note that the only possible value of $\word{v}_{i + 1}$ is $0$ following the arguments laid out in Lemma \ref{lem:i_le_j_beta}.
Therefore value of $SB(\word{w}, p_f, p_b, \word{B_f}, \word{B_b}, i, j,r)$ is equal to $SB(\word{w}, p_f + 1, p_b + 1, WX[\word{B_f},0], WX[\word{B_b},1], i + 1, j,r)$.

In the fifth case, the value of $\word{v}_{i + 1}$ can be either $0$ or $1$ for every $\word{v} \in \mathbf{B}(\word{w}, p_f, p_b, \word{B_f}, \word{B_b}, j, j,r)$.
If $\word{v}_{i + 1} = 0$, then as $\word{v}_{[i - p_f,i + 1]} < \word{w}_{[1,p_f + 1]}$, the values of the remain symbols in $\word{v}$ can be determined as in the third case, giving a total of $|\mathbf{Y}(\word{w}, r - i - 1, p_b + 1, WX[\word{B_f}, 0])|$ possible values.
Following the same arguments as in Lemma \ref{lem:i_le_j_beta}, the number of words in $\mathbf{B}(\word{w}, p_f, p_b, \word{B_f}, \word{B_b}, j, j,r)$ where the symbol at position $i + 1$ is 1 is $SB(\word{w}, p_f + 1, p_b + 1, WX[\word{B_f},1], WX[\word{B_b},0], i + 1, j,r)$.

In the sixth case, following the same arguments as in Lemma \ref{lem:i_le_j_beta}, the value of $SB(\word{w}, p_f, p_b, \word{B_f}, \word{B_b}, i, j,r)$ is given by $SB(\word{w}, p_f + 1, 0, WX[\word{B_f},0], WX[\word{B_b},1], i + 1, j,r)$.
\end{proof}

\begin{lemma}
\label{lem:Y_size_of}
The size of $\mathbf{Y}(\word{w},i,p_b,\word{B_f})$ is equal to:
\begin{itemize}
    \item $0$ if $i = 0$ and $\word{w}_{[1,p_b]} : \word{B_f} < \word{w}_{[1,p_b + |\word{B_f}|]}$.
    \item $1$ if $i = 0$ and $\word{w}_{[1,p_b]} : \word{B_f} \geq \word{w}_{[1,p_b + |\word{B_f}|]}$.
    \item $\mathbf{Y}(\word{w},i - 1,p_b + 1,WX[\word{B_f},1])$ if $i > 0$ and $\word{w}_{p_b + 1} = 1$.
    \item $\mathbf{Y}(\word{w},i - 1,p_b + 1,WX[\word{B_f},0]) + \mathbf{Y}(\word{w},i - 1,0,WX[\word{B_f},0])$ if $i > 0$ and $\word{w}_{p_b + 1} = 0$.
\end{itemize}
\end{lemma}

\begin{proof}
In the first case, the only value for $\word{v} \in \mathbf{Y}(\word{w},i,p_b,\word{B_f})$ is the empty word $\emptyset$.
As $\word{w}_{[1,p_b]} : \word{B_f} < \word{w}_{[1,p_b + |\word{B_f}|]}$, the size of $\mathbf{Y}(\word{w},i,p_b,\word{B_f})$ must be 0.
Similarly, in the second case, the only value for $word{v} \in \mathbf{Y}(\word{w},i,p_b,\word{B_f})$ is the empty word $\emptyset$.
However, as $\word{w}_{[1,p_b]} : \word{B_f} \geq \word{w}_{[1,p_b + |\word{B_f}|]}$, the empty word is a valid choice for $\mathbf{Y}(\word{w},i,p_b,\word{B_f})$ and hence $|\mathbf{Y}(\word{w},i,p_b,\word{B_f})|= 1$.

In the third case, following the same arguments as in Lemma \ref{lem:i_le_j_beta}, the only possible value for $\word{v}_1$ is $1$, therefore the size of $\mathbf{Y}(\word{w},i,p_b,\word{B_f})$ equals the size of $\mathbf{Y}(\word{w},i - 1,p_b + 1,WX[\word{B_f},1])$.
Finally, in the fourth case, the value of $\word{v}_1$ can either be $0$ or $1$.
In the case $\word{v}_1 = 0$, then following the same arguments as in Lemma \ref{lem:i_le_j_beta}, there are $\mathbf{Y}(\word{w},i - 1,p_b + 1,WX[\word{B_b},0])$ possible suffixes.
Similarly, if $\word{v}_1 = 1$, then there are $\mathbf{Y}(\word{w},i - 1,0,WX[\word{B_b},0])$ possible suffixes.
Hence the total number of possible values of $\word{v}$ is $\mathbf{Y}(\word{w},i - 1,p_b + 1,WX[\word{B_b},0]) + \mathbf{Y}(\word{w},i - 1,p_b + 1,WX[\word{B_b},0])$.
\end{proof}

The above provides an outline allowing for counting the number of words $\word{v} \in \beta(\word{w},j,r)$ where $\word{v}_{[1,i]} \not\sqsubseteq \word{w}$ and $S(\word{v}_{[1,i]}) \not\sqsubseteq \word{w}$ for every $i \in [r]$.
The number of words where either $\word{v}_{[1,i]} \sqsubseteq \word{w}$ or $S(\word{v}_{[1,i]}) \sqsubseteq \word{w}$, but $\word{v}_{[1,i + 1]} \not\sqsubseteq \word{w}$ and $S(\word{v}_{[1,i + 1]}) \not\sqsubseteq \word{w}$ equals the value of $SB(\word{w}, p_f, p_b, \word{B_f}, \word{B_b}, i + 1, j,r)$, where:
\begin{itemize}
    \item $p_f$ is the length of the longest suffix of $\word{v}_{[1,i + 1]}$ that is also a prefix of $\word{w}$, i.e. the largest value such that $\word{v}_{[i + 1 - p_f,i + 1]} = \word{w}_{[1,p_f]}$.
    \item $p_b$ is the length of the longest suffix of $S(\word{v}_{[1,i + 1]})$ that is also a prefix of $\word{w}$, i.e. the largest value such that $S(\word{v}_{[i + 1 - p_b,i + 1]}) = \word{w}_{[1,p_b]}$.
    \item $\word{B_f}$ is the subword of $\word{w}$ that bounds $\word{v}_{[1,i + 1]}$.
    \item $\word{B_b}$ is the subword of $\word{w}$ that bounds $S(\word{v}_{[1,i + 1]})$.
\end{itemize}
Note that if $i \geq j$, it is necessary to check that $\word{v}_{[1,i + 1]}$ is a valid prefix following the same arguments in Lemmas \ref{lem:i_e_j_beta} and \ref{lem:i_ge_j_beta}.
Using these observations, the size of $\beta(\word{w},j,r)$ can be computed in $O(n^5)$ time.

\begin{lemma}
\label{lem:counting_beta}
The size of $\beta(\word{w},j,r)$ can be computed in $O(n^5)$ time.
\end{lemma}

\begin{proof}
Note first that the size of $\mathbf{Y}(\word{w},i,p_b,\word{B_f})$ can be computed in $O(n^4)$ time using the dynamic algorithm implicitly given in Lemma \ref{lem:Y_size_of} for every $i,p_b \in [r]$ and $\word{B_f} \sqsubseteq \word{w}$.
The computation in this case is done by starting with $p_b = i = 0$ and $\word{B_f} = \emptyset$, then progressing by increasing the value of $p_b$ first, then the length of the subwords $\word{B_f}$ and finally the value of $i$.
Using the dynamic algorithm implicitly given above, note that the value of $SB(\word{w}, p_f, p_b, \word{B_f}, \word{B_b}, i, j,r)$ can be computed in $O(r^5)$ provided the value of $\mathbf{Y}(\word{w},i,p_b,\word{B_f})$ has been precomputed for every  $i,p_b \in [r]$ and $\word{B_f} \sqsubseteq \word{w}$.
In this case, we compute the value of $SB(\word{w}, p_f, p_b, \word{B_f}, \word{B_b}, i , j,r)$ starting with $i = r, \word{B_f}, \word{B_b} \sqsubseteq_i \word{w}$, $p_b,p_f = 0$, and progressing in increasing value of $p_b$ and $p_f$, then decreasing the value of $i$ while maintaining the size of $\word{B_f}$ and $\word{B_b}$ as being the size of $i$.
Finally, each of the $n^2$ possible subwords of $\word{w}$ can be checked as a potential prefixes of words in $\beta(\word{w},j,r)$.

Let $\word{u} \sqsubseteq_{i - 1} \word{w}$ be a subword of $\word{w}$ of length $i - 1$ such that $ i \leq r - 1$ and let $x \in \{0,1\}$ be a symbol such that $\word{u} : x \not\sqsubseteq \word{w}$.
Further let $\word{B_f} \sqsubseteq_i \word{w}$ be the word bounding $\word{u} : x$ and let $\word{B_b} \sqsubseteq_i \word{w}$ be the subword bounding $S(\word{u} : x)$.
Similarly, let $p_f$ be the length of the longest suffix of $\word{u} : x$ that is a prefix of $\word{w}$, and let $p_b$ be the length of the longest suffix of $S(\word{u} : x)$ that is a prefix of $\word{w}$, i.e. the maximum values for which $(\word{u} : x)_{i - p_f, i} = \word{w}_{[1,p_f]}$ and $S(\word{u} : x)_{i - p_b, i} = \word{w}_{[1,p_b]}$.
Observe that $\word{u} : x$ is the prefix of some word in $\beta(\word{w},j,r)$ if and only if one of the following hold:
\begin{itemize}
    \item if $i < j$ then $(\word{u} : x)_{[i - s, i]} \geq \word{w}_{[1,s]}$ and $S(\word{u} : x)_{i - s, i} \geq \word{w}_{[1,s]}$ for every $s \in [i]$.
    \item if $i = j$ then $p_f = 0$ and $S(\word{u} : x)_{i - s, i} \geq \word{w}_{[1,s]}$ for every $s \in [i]$.
    \item if $r > i > j$ then $p_f = i - j$ and  $S(\word{u} : x)_{i - s, i} \geq \word{w}_{[1,s]}$ for every $s \in [i]$.
\end{itemize}
Provided $\word{u} : x$ satisfies one of the above conditions, following corollary \ref{col:beta_cartesian} the number of words in $\beta(\word{w},r,j)$ where $\word{u}:x$ is the prefix is $SB(\word{w}, p_f, p_b, \word{B_f}, \word{B_b}, i, j)$.
As there are at most $n^2$ such subwords of $\word{w}$ and each condition can be checked in at most $O(n)$, the size of $\beta(\word{w}, r,j)$ can be found in $O(n^5 + n^3) = O(n^5)$ time.
\end{proof}

\begin{theorem}
\label{thm:symmetric}
The value of $RankSN(\word{w},m)$ can be computed in $O(n^6 \log^2 n)$ time for any $m \leq n$.
\end{theorem}

\begin{proof}
Following Lemmas \ref{lem:RA_to_RB}, \ref{lem:rank_sl_from_RB}, and \ref{lem:rank_sn_from_sl}, the value of $RankSN(\word{w},m,r)$ is:
$$
RankSN(\word{w},m,r) = \sum\limits_{d | r} \left( \frac{1}{2 \cdot r}\sum\limits_{p | d} \mu\left(\frac{d}{p}\right) |\mathbf{RA}(\word{w},m,S,p)| \right)
$$
\noindent
From Observations \ref{obs:alpha_mapping} and \ref{obs:beta_mapping} , the size of $\mathbf{RA}(\word{w},m,S,r)$ equals $\sum\limits_{j \in [m]} |\alpha(\word{w},r,j)| + |\beta(\word{w},r,j)|$.
Following Lemma \ref{lem:counting_alpha}, the size of $\alpha(\word{w},r,j)$ can be computed in $O(n^3)$ time.
Following Lemma \ref{lem:counting_beta}, the size of $\beta(\word{w},r,j)$ can be computed in $O(n^5)$ time.
As there are at most $O(n)$ values of $j$, the total time complexity for determining the size of $\mathbf{RA}(\word{w},m,S,r)$ is $O(n^6)$.
As there are at most $O(\log n)$ possible divisors of $r$, the size of $\mathbf{RA}(\word{w},m,S,p)$ needs to be evaluated at most $O(\log n)$ times, giving a total time complexity of $O(n^6 \log n)$.
The value of $RankSN(\word{w},m,r)$ can then be computed in at most $O(\log^{2} n)$ time once the size of $\mathbf{RA}(\word{w},m,S,p)$ has been precomputed for every factor $p$ of $r$.
Finally, following Lemma \ref{lem:rank_sn_general}, the value of $RankSN(\word{w},m)$ can be computed from the value of $RankSN(\word{w},m,r)$ for at most $O(\log n)$ values of $r$.
Therefore the total time complexity of computing $RankSN(\word{w},m)$ is $O(n^6 \log^2 n)$.
\end{proof}

\section{Enclosing Necklaces}
\label{sec:enclosing}

This section shows how to rank a word $\word{w}$ within the set of binary unlabelled necklaces enclosing $\word{w}$.
Note that the rank of $\word{w}$ within this set is equivalent to the number of binary unlabelled necklaces enclosing $\word{w}$.
As with the ranking approach to symmetric necklaces, we start with the key theoretical results that inform our approach.

\begin{lemma}
\label{lem:ranking_enclosing_words}
Let $RankEN(\word{w},m)$ be the rank of $\word{w}$ within the set of necklaces of length $m$ that enclose $\word{w}$ and let $RankEL(\word{w},m)$ be the rank of $\word{w}$ within the set of Lyndon words of length $m$ that enclose $\word{w}$.
$RankEN(\word{w},m) = \sum\limits_{d | m} RankEL(\word{w},d)$.
\end{lemma}

\begin{proof}
Observe that every necklace counted by $RankEN(\word{w},m)$ must have a unique period that is a factor of $m$, hence $RankEN(\word{w},m) = \sum\limits_{d | m} RankEL(\word{w},d)$.
\end{proof}

\begin{lemma}
\label{lem:EL_to_rank}
Let $\mathbf{EL}(\word{w},m)$ be the set of words of length $m$ belonging to a Lyndon word that encloses $\word{w}$.
$RankEL(\word{w},m) = \frac{|\mathbf{EL}(\word{w},m))}{m}$.
\end{lemma}

\begin{proof}
Following the same arguments as in Lemma \ref{lem:rank_sl_from_RB}, every aperiodic necklace counted by $RankEL(\word{w},m)$ must have exactly $m$ words in $\mathbf{EL}(\word{w},m)$ representing it.
Therefore $RankEL(\word{w},m) = \frac{|\mathbf{EL}(\word{w},m))}{m}$.
\end{proof}

\begin{lemma}
\label{lem:EL_from_words}
Let $\mathbf{EL}(\word{w},m)$ be the set of words of length $m$ belonging to a Lyndon word that encloses $\word{w}$ and let $\mathbf{EN}(\word{w},m)$ be the set of words of length $m$ belonging to a necklace that encloses $\word{w}$.
The size of $\mathbf{EL}(\word{w},m)$ equals $\sum\limits_{d | m}  \mu\left(\frac{m}{d}\right)|\mathbf{EN}(\word{w}, d)| $.
\end{lemma}

\begin{proof}
Following the same arguments as in Lemma \ref{lem:ranking_enclosing_words}, the size of $\mathbf{EN}(\word{w},m)$ can be expressed in terms of the size of $\mathbf{EL}(\word{w},d)$ for every factor $d$ of $m$ as $|\mathbf{EN}(\word{w},m)| = \sum\limits_{d | m} |\mathbf{EL}(\word{w},d)|$.
Applying the M\"obius inversoin formula to this equation gives $|\mathbf{EL}(\word{w},m)| = \sum\limits_{d | m}  \mu\left(\frac{m}{d}\right)|\mathbf{EN}(\word{w}, d)|$.
\end{proof}

\noindent
As in the Symmetric case, we partition the set of necklaces into a series of subsets for ease of computation.
Let $\gamma(\word{w},m,r)$ denote the set of words belonging to a necklace which encloses $\word{w}$ such that $r$ is the smallest rotation for which $\word{v} \in \gamma(\word{w},m,r)$ is smaller than $\word{w}$, i.e. the smallest value where $\Angle{\word{v}}_r < \word{w}$.
We further introduce the set $\mathbf{C}(\word{w},i,r,\word{B_f}, \word{B_b}, p_f, p_b) \subseteq \gamma(\word{w},m,r)$ as the set of words where every $\word{v} \in \mathbf{C}(\word{w},i,r,\word{B_f}, \word{B_b}, p_f, p_b)$ satisfies the following conditions:
\begin{enumerate}
    \item \label{cond:enc_small_before_s} $\Angle{\word{v}}_{s} > \word{w}$ for every $s \in [r - 1]$.
    \item \label{cond:bigger_in_every_rotation} $\Angle{S(\word{v})}_s > \word{w}$ for every $s \in [m]$.
    \item \label{cond:bigger_eventually} $\Angle{\word{v}}_r < \word{w}$.
    \item \label{cond:bounding_forward_enclosing} $\word{v}_{[1,i]}$ is bound by $\word{B_f} \sqsubseteq_i \word{w}$.
    \item \label{cond:bounding_backwords_enclosing} $S(\word{v}_{[1,i]})$ is bound by $\word{B_b} \sqsubseteq_i \word{w}$.
    \item \label{cond:} $p_f$ is the length of the longest suffix of $\word{v}_{[1,i]}$ that is a prefix of $\word{w}$, i.e. the largest value such that $\word{v}_{[i - p_f, i]} = \word{w}_{[1,p_f]}$.
    \item $p_b$ is the length of the longest suffix of $S(\word{v}_{[1,i]})$ that is a prefix $\word{w}$, i.e. the largest value such that $S(\word{v}_{[i - p_b, i]}) = \word{w}_{[1,p_b]}$.
\end{enumerate}
Note that Conditions \ref{cond:enc_small_before_s}, \ref{cond:bigger_in_every_rotation}, and \ref{cond:bigger_eventually} are the necessary conditions for $\word{v}$ to be in $\gamma(\word{w},m,r)$.
As before, we break our dynamic programming based approach into several sub cases based on the value of $i$ relative to $r$.
As in the symmetric case, we relay upon a technical proposition.

\begin{proposition}
\label{prop:enclosing_structure}
Given $\word{v} \in \mathbf{C}(\word{w},i,r,\word{B_f}, \word{B_b}, p_f, p_b)$, $\word{v}$ also belongs to $\mathbf{C}(\word{w},i + 1,r,WX[\word{B_f},\word{v}_{i + 1}], WX[\word{B_b},\word{v}_{i + 1}], p_f', p_b') $
\end{proposition}

\begin{corollary}
\label{col:enclosing_cartesian}
Given a pair of words $\word{v},\word{u} \in \mathbf{C}(\word{w},i,r,\word{B_f}, \word{B_b}, p_f, p_b) $ let $\word{v}' = \word{v}_{[1,i]} : \word{u}_{[i + 1, m]}$.
Then $\word{v}' \in \mathbf{C}(\word{w},i + 1,r,\word{B_f}', \word{B_b}', p_f', p_b')$ if and only if $\word{v} \in \mathbf{C}(\word{w},i + 1,r,\word{B_f}', \word{B_b}', p_f', p_b')$.
\end{corollary}

\begin{lemma}
\label{lem:enclosing_i_le_r}
If $i < r$, then the size of $\mathbf{C}(\word{w},i,r,\word{B_f}, \word{B_b}, p_f, p_b)$ is either:
\begin{itemize}
    \item $0$ if $\word{w}_{p_f + 1} = \word{w}_{p_b + 1} = 1$.
    \item $|\mathbf{C}(\word{w},i+1,r,WX[\word{B_f},0] WX[\word{B_b},1], p_f + 1, p_b + 1)|$ if $\word{w}_{p_f + 1} = 0$ and $\word{w}_{p_b + 1} = 1$.
    \item $|\mathbf{C}(\word{w},i+1,r,WX[\word{B_f},1], WX[\word{B_b},0], p_f + 1, p_b + 1)|$ if $\word{w}_{p_f + 1} = 1$ and $\word{w}_{p_b + 1} = 0$.
    \item $|\mathbf{C}(\word{w},i + 1,r,WX[\word{B_f},0], WX[\word{B_b},1], p_f + 1, 0)| + |\mathbf{C}(\word{w},i,r,WX[\word{B_f},1], WX[\word{B_b},0], 0, p_b + 1)|$ if $\word{w}_{p_f + 1} = \word{w}_{p_b + 1} = 0$.
\end{itemize}
\end{lemma}

\begin{proof}
In the first case, observe that if the symbol in position $i + 1$ of any word $\word{v} \in \mathbf{C}(\word{w},i,r,\word{B_f}, \word{B_b}, p_f, p_b)$ is 0 then $\word{v}_{[i - p_f, i + 1]} < \word{w}_{[1,p_f + 1]}$, violating Condition \ref{cond:enc_small_before_s}.
Similarly if $\word{v}_{i + 1} = 1$ then $S(\word{v}_{[i - p_b, i + 1]}) < \word{w}_{[1,p_b + 1]}$, violating Condition \ref{cond:bigger_in_every_rotation}.

In the second case, if $\word{v}_{i + 1} = 1$ then $S(\word{v}_{[i - p_b, i + 1]}) < \word{w}_{[1,p_b + 1]}$, violating Condition \ref{cond:bigger_in_every_rotation}.
Therefore $\word{v}_{i + 1} = 0$.
As such any word in $\mathbf{C}(\word{w},i+1,r,\word{B_f}, \word{B_b}, p_f, p_b)$ must also be in $\mathbf{C}(\word{w},i,r,WX[\word{B_f},0], WX[\word{B_b},1], p_f + 1, p_b + 1)$.

Similarly in the third case if $\word{v}_{i + 1} = 0$ then $\word{v}_{[i - p_f, i + 1]} < \word{w}_{[1,p_f + 1]}$, violating Condition \ref{cond:enc_small_before_s}.
Therefore $\word{v}_{i + 1} = 1$.
As such any word in $\mathbf{C}(\word{w},i+1,r,\word{B_f}, \word{B_b}, p_f, p_b)$ must also be in $\mathbf{C}(\word{w},i,r,WX[\word{B_f},1], WX[\word{B_b},0], p_f + 1, p_b + 1)$.

In the final case, the value of $\word{v}_{i + 1}$ can be either $0$ or $1$.
If $\word{v}_{i + 1} = 0$ then the prefix of $\word{w}$ starting at position $i - p_f$ is continued, while the prefix in the relabelling is not, giving a total of $|\mathbf{C}(\word{w},i + 1,r,WX[\word{B}_f,0], WX[\word{B}_b,1], p_f + 1, 0)|$ words.
If $\word{v}_{i + 1} = 1$ then the prefix in the relabelling is continued, while the original prefix is not, giving a total of $|\mathbf{C}(\word{w},i,r,WX[\word{B}_f,1], WX[\word{B}_b,0], 0, p_b + 1)|$ words.
Therefore the total size of $\mathbf{C}(\word{w},i,r,\word{B_f}, \word{B_b}, p_f, p_b)$ in this case is $|\mathbf{C}(\word{w},i + 1,r,WX[\word{B}_f,0], WX[\word{B}_b,1], p_f + 1, 0)| + |\mathbf{C}(\word{w},i,r,WX[\word{B}_f,1], WX[\word{B}_b,0], 0, p_b + 1)|$.
\end{proof}

\begin{lemma}
\label{lem:enclosing_equals}
If $i = r$ then the size of $\mathbf{C}(\word{w},i,r,\word{B_f}, \word{B_b}, p_f, p_b)$ is:
\begin{itemize}
    \item 0 if $p_f > 0$.
    \item $|\mathbf{C}(\word{w},i,r,WX[\word{B_f},0], WX[\word{B_b},1], 1, 0)|$ if $\word{w}_{p_b} = 0$.
    \item $|\mathbf{C}(\word{w},i,r,WX[\word{B_f},0], WX[\word{B_b},1], 1, p_b + 1)|$ if $\word{w}_{p_b} = 1$.
\end{itemize}
\end{lemma}

\begin{proof}
Note that $\word{w}_1 = 0$.
Therefore, $\word{v}_{i + 1} = 0$.
In the first case if $\word{v}_ {[i - p_f, i]} \neq \emptyset$ then as $\Angle{\word{v}}_{r} < \word{w}$ then as $\word{v}_{[1,m]} < \word{w}_{[p_f + 1, m]}$ $\Angle{\word{v}}_{r - p_f} < \word{w}$.
Hence there can be no possible words in $\mathbf{C}(\word{w},i,r,\word{B_f}, \word{B_b}, p_f, p_b)$ when $p_f > 0$.

In the second case, as $\word{v}_{i + 1} = 0$, $S(\word{v}_{i + 1}) = 1$ and hence the prefix is not continued in the relabelled word.
This leaves $|\mathbf{C}(\word{w},i,r,WX[\word{B_f},0], WX[\word{B_b},1], 1, 0)|$ possible words in $\mathbf{C}(\word{w},i,r,\word{B_f}, \word{B_b}, p_f, p_b)$.

Similarly in the third case as $S(\word{v}_{i + 1}) = 1$ the prefix is continued in the relabelled word, giving a total of $|\mathbf{C}(\word{w},i,r,WX[\word{B_f},0], WX[\word{B_b},1], 1, p_b + 1)|$ words in $\mathbf{C}(\word{w},i,r,\word{B_f}, \word{B_b}, p_f, p_b)$.
\end{proof}

\noindent
Before we introduce the penultimate step, we need an auxiliary set.
The set $\mathbf{Y}(\word{w}, p_b, \word{B}_b, i)$ is introduced as the set of words of length $i$ where $\word{w}_{[1,p_b]} : \word{v} : \word{B}$ for every $\word{v} \in \mathbf{Y}(\word{w}, p_b, \word{B}_b, i)$.
The set $\mathbf{Y}(\word{w}, p_b, \word{B}_b, i)$ is treated as a black box for now.

\begin{lemma}
\label{lem:enclosing_i_ge_j}
If $m > i > r$ then the size of $\mathbf{C}(\word{w},i,r,\word{B_f}, \word{B_b}, p_f, p_b)$ is:
\begin{itemize}
    \item $|\mathbf{Y}(\word{w}, p_b + 1, WX[\word{B}_b,0], m - i - 1)|$ if $\word{w}_{p_f + 1} = \word{w}_{p_b + 1} = 1$.
    \item $|\mathbb{B}(\word{w}, i + 1, r, WX[\word{B_f},0], WX[\word{B_f},1], p_f + 1, p_b + 1)|$ if $\word{w}_{p_f + 1} = 0$ and $\word{w}_{p_b + 1} = 1$.
    \item $|\mathbf{Y}(\word{w}, 0, WX[\word{B}_b,1], m - i - 1)| + |\mathbb{B}(\word{w}, i + 1, r, WX[\word{B_f},1], WX[\word{B_f},0], p_f + 1, p_b + 1)|$ if $\word{w}_{p_f + 1} = 1$ and $\word{w}_{p_b + 1} = 0$.
    \item $|\mathbb{B}(\word{w}, i + 1, r, WX[\word{B_f},0], WX[\word{B_f},1], p_f + 1, 0)|$ if $\word{w}_{p_f + 1} = \word{w}_{p_b + 1} = 0$.
\end{itemize}
\end{lemma}

\begin{proof}
In the first case, the only possible value of $\word{v}_{i + 1}$ is $0$.
Therefore $\word{v}_{[i - p_f, i + 1]} < \word{w}_{[1, p_f + 1]}$ and hence Condition \ref{cond:bigger_eventually} is satisfied for any possible suffix.
The number of suffixes satisfying Condition \ref{cond:bigger_in_every_rotation} is therefore given by $|\mathbf{Y}(\word{w}, p_b + 1, WX[\word{B}_b,0], m - i - 1)|$.

In the second case, the only possible value of $\word{v}_{i + 1}$ is $0$.
As this continues the prefixes in both the the original version and the relabelled version the size of $\mathbf{C}(\word{w},i,r,\word{B_f}, \word{B_b}, p_f, p_b)$ in this case is equal to the size of $\mathbb{B}(\word{w}, i + 1, r, WX[\word{B_f},0], WX[\word{B_f},1], p_f + 1, p_b + 1)$.

In the third case the value of $\word{v}_{i + 1}$ can be either $0$ or $1$.
If $\word{v}_{i + 1} = 0$ then $\word{v}_{[i - p_f, i + 1]} < \word{w}_{[1, p_f + 1]}$ and hence Condition \ref{cond:bigger_eventually} is satisfied for any possible suffix.
The number of suffixes satisfying Condition \ref{cond:bigger_in_every_rotation} is therefore given by $|\mathbf{Y}(\word{w}, 0, WX[\word{B}_b,1], m - i - 1)|$.
Similarly the number of words in where symbol $i + 1$ is 0 is equal to the size of the set $\mathbb{B}(\word{w}, i + 1, r, WX[\word{B_f},1], WX[\word{B_f},0], p_f + 1, p_b + 1)$.

In the fourth case, the only possible value of $\word{v}_{i + 1}$ is 0.
As $\word{v}_{[i - p_f, i + 1]} = \word{w}_{[1,p_f + 1]}$ the number of words in $\mathbf{C}(\word{w},i,r,\word{B_f}, \word{B_b}, p_f, p_b)$ equals the size of the set $\mathbb{B}(\word{w}, i + 1, r, WX[\word{B_f},0], WX[\word{B_f},1], p_f + 1, 0)$.
\end{proof}

\begin{lemma}
\label{lem:enclosing_i_eq_m}
If $i = m$, then the size of $\mathbf{C}(\word{w},i,r,\word{B_f}, \word{B_b}, p_f, p_b)$ is either:
\begin{itemize}
    \item 0 if $\word{w}{[1,p_f]} : \word{B_f}_{[1,m - p_f]} \geq \word{w}$ or $\word{w}_{[1,p_b]} : \word{B_b}{[1,m - p_b]} < \word{w}$.
    \item 1 if $\word{w}{[1,p_f]} : \word{B_f}_{[1,m - p_f]} < \word{w}$ and $\word{w}_{[1,p_b]} : \word{B_b}{[1,m - p_b]} \geq \word{w}$.
\end{itemize}
\end{lemma}

\begin{proof}
If $\word{w}{[1,p_f]} : \word{B_f}_{[1,m - p_f]} \geq \word{w}$, then as $\word{v} > \word{B}_f$, $\word{w}{[1,p_f]} : \word{v}_{[1,m - p_f]} > \word{w}$.
Therefore, there exists no such word in the set $\mathbf{C}(\word{w},i,r,\word{B_f}, \word{B_b}, p_f, p_b)$.
Similarly if $\word{w}_{[1,p_b]} : \word{B_b}{[1,m - p_b]} < \word{w}$, the $\Angle{S(\word{v})}_{m - p_f} < \word{w}$ contradicting Condition \ref{cond:bigger_in_every_rotation}.

On the other hand, if $\word{w}{[1,p_f]} : \word{B_f}_{[1,m - p_f]} < \word{w}$, then as $\word{B_f}_{[1,m - p_f]} < \word{w}_{[p_f, m]}$, by the definition of strictly bounding subwords $\word{v}_{[1,m - p_f]} < \word{w}_{[p_f, m]}$.
Similarly, as in the first case, $\word{w}_{[1,p_b]} : \word{B_b}{[1,m - p_b]}$ must be greater than or equal to $\word{w}$ as otherwise there is a contradiction with Condition \ref{cond:bigger_in_every_rotation}.
As there is no possible extension of $\word{v}$ other than the empty word.
\end{proof}

\noindent
Finally, we show how to compute the size of $\mathbf{Y}(\word{w}, p_b, \word{B}_b, i)$.
\begin{lemma}
\label{lem:computing_Y}
The size of $\mathbf{Y}(\word{w}, p_b, \word{B_b}, i)$ is:
\begin{itemize}
    \item 0 if $i = 0$ and $\word{w}_{[1,p_b]} : \word{B_b} < \word{w}$.
    \item 1 if $i = 0$ and $\word{w}_{[1,p_b]} : \word{B_b} \geq \word{w}$.
    \item $|\mathbf{Y}(\word{w}, p_b + 1, WX[\word{B_b},1], i - 1)|$ if $\word{w}_{p_b + 1} = 1$ and $i > 0$.
    \item $|\mathbf{Y}(\word{w}, p_b + 1, WX[\word{B_b},0], i - 1)| + |\mathbf{Y}(\word{w}, 0, WX[\word{B_b},1], i - 1)|$ if $\word{w}_{p_b + 1} = 0$ and $i > 0$.
\end{itemize}
\end{lemma}

\begin{proof}
In the first case, as $i = 0$ the only possible word in $\mathbf{Y}(\word{w}, p_b, \word{B}_b, i)$ is the empty word $\emptyset$.
As $\word{w}_{[1,p_b]} : \word{B_b} < \word{w}$, this contradicts the definition of $\mathbf{Y}(\word{w}, p_b, \word{B}_b, i)$.
Therefore, the size of $\mathbf{Y}(\word{w}, p_b, \word{B}_b, i)$ must be 0.
In the second case, as $\word{w}_{[1,p_b]} : \word{B_b} \geq \word{w}$, the empty word is a valid choice for $\word{w}_{[1,p_b]} : \word{B_b} \geq \word{w}$.
Therefore, the size of $\mathbf{Y}(\word{w}, p_b, \word{B}_b, i)$ is 1.

When $i > 0$, then there are two possibilities based on the value of $\word{w}_{p_b + 1}$.
If $\word{w}_{p_b + 1} = 1$ then the first symbol of every word in $\mathbf{Y}(\word{w}, p_b, \word{B}_b, i)$ must be 1.
Therefore, the number of words in $\mathbf{Y}(\word{w}, p_b, \word{B}_b, i)$ is equal to the number of possible suffixes of the words in $\mathbf{Y}(\word{w}, p_b, \word{B}_b, i)$ of length $i - 1$, which is equal to $|\mathbf{Y}(\word{w}, p_b + 1, WX[\word{B}_b,1], i - 1)|$.
If $\word{w}_{p_b + 1} = 0$ then the first symbol of every word in $\mathbf{Y}(\word{w}, p_b, \word{B}_b, i)$ can be either $0$ or $1$.
The number of possible suffixes of the words in $\mathbf{Y}(\word{w}, p_b, \word{B}_b, i)$ of length $i - 1$ when the first symbol is $0$ is is $|\mathbf{Y}(\word{w}, p_b + 1, WX[\word{B}_b,0], i - 1)|$.
On the other hand, the number of possible suffixes of the words in $\mathbf{Y}(\word{w}, p_b, \word{B}_b, i)$ of length $i - 1$ with the first symbol $1$ is $|\mathbf{Y}(\word{w}, 0, WX[\word{B}_b,1], i - 1)|$.
Therefore the total number of words in  $\mathbf{Y}(\word{w}, p_b, \word{B}_b, i)$ is $|\mathbf{Y}(\word{w}, p_b + 1, WX[\word{B_b},0], i - 1)| + |\mathbf{Y}(\word{w}, 0, WX[\word{B_b},1], i - 1)|$.
\end{proof}

\begin{lemma}
\label{lem:counting_gamma}
The number of words in $\gamma(\word{w},m,r)$ can be computed in $O(n^5)$ time.
\end{lemma}

\begin{proof}
This argument follows the same outline as Lemma \ref{lem:counting_alpha}.
First, consider the value of $\mathbf{Y}(\word{w}, p_b, \word{B_b}, i)$.
Note that there are at most $n$ possible values of $p_b$ and $i$, and $n^2$ possible values of $\word{B_b} \sqsubseteq \word{w}$.
Following the dynamic programming approach outlined by Lemma \ref{lem:computing_Y}, the size of  $\mathbf{Y}(\word{w}, p_b, \word{B_b}, i)$ can be computed in at most $n$ time if $i = 0$ or in $O(1)$ operation $\mathbf{Y}(\word{w}, p_b', \word{B_b}', i - 1)$ has been computed for every $p_b' \in \{p_b + 1, 0\}$ and $\word{B_b}' \sqsubseteq \word{w}$.
As there are at most $O(n^3)$ possible values of $p_b \in [n]$ and $\word{B_b} \sqsubseteq \word{w}$, therefore every value of $\mathbf{Y}(\word{w}, p_b, \word{B_b}, 0)$ can be computed in $O(n^4)$ time.
Further, by starting with $i = 0$ and computing for increasing value of $i$, the values of each $\mathbf{Y}(\word{w}, p_b, \word{B_b}, i)$ for $i > 0$ can be computed in $O(1)$ time per value, requiring $O(n^4)$ time in total.

Moving to the problem of computing the size of $\mathbf{C}(\word{w},i,r,\word{B_f}, \word{B_b}, p_f, p_b)$.
From Lemma \ref{lem:enclosing_i_eq_m}, the size of $\mathbf{C}(\word{w},i,r,\word{B_f}, \word{B_b}, p_f, p_b)$ can be computed in at most $O(n)$ operations if $i = m$.
Note that there are at most $O(n^4)$ possible values of $p_f, p_b \in [n], \word{B_f}, \word{B_b} \sqsubseteq_i \word{w}$.
Therefore, the value of $\mathbf{C}(\word{w},i,r,\word{B_f}, \word{B_b}, p_f, p_b)$ for every $i = m,p_f, p_b \in [n], \word{B_f}, \word{B_b} \sqsubseteq_i \word{w}$ can be computed in $O(n^5)$ time.

Using Lemmas \ref{lem:enclosing_i_le_r}, \ref{lem:enclosing_equals}, and \ref{lem:enclosing_i_ge_j}, the size of $\mathbf{C}(\word{w},i,r,\word{B_f}, \word{B_b}, p_f, p_b)$ can be computed in $O(1)$ time from the values of $\mathbf{Y}(\word{w}, p_b', \word{B_b}', m - i - 1)$ and $\mathbf{C}(\word{w},i + 1,r,\word{B_f}', \word{B_b}', p_f', p_b')$.
Therefore, the size of $\mathbf{C}(\word{w},i,r,\word{B_f}, \word{B_b}, p_f, p_b)$ can be computed for every $i,p_f,p_b \in [n], \word{B_f}, \word{B_b} \sqsubseteq_i \word{w}$.

The final case to consider are words where some prefix is a subword of $\word{w}$.
This can be done by looking at every word $\word{v} \sqsubseteq \word{w}$ and determining if $\word{v}$ can be used as a prefix for some word in $\gamma(\word{w},m,r)$.
Let $\word{v} \sqsubseteq_i \word{w}$ be a subword of $\word{w}$ and let $x \in \{0,1\}$ by some symbol such that $\word{v} : x \not\subseteq \word{w}$.
Further, let $p_f$ be the length of the longest suffix of $\word{v} : x$ that is a prefix of $\word{w}$, i.e. the largest value such that $(\word{v} : x)_{[i + 1 - p_f, i + 1]} = \word{w}_{[1,p_f]}$.
Similarly let $p_b$ be the length of the longest suffix of $S(\word{v}:x)$ that is a prefix of $\word{w}$, i.e. the largest value such that $S(\word{v} : x)_{[i + 1 - p_b, i + 1]} = \word{w}_{[1,p_b]}$.
$\word{v}$ is the prefix of some word in $\gamma(\word{w},m,r)$ if and only if there exists some symbol $x \in \{0,1\}$ where:
\begin{itemize}
    \item For every $s \in [i + 1 - p_b]$, $S(\word{v} : x)_{[i + 1 - p_b, i + 1]} > \word{w}_{[1,s]}$.
    \item If $r = 0$, then $\word{v} : x < \word{w}$.
    \item If $i + 1 \leq r$ then $\word{v} : x < \word{w}$.
    \item If $i + 1 > r$ then $p_f = i + 1 - r$.
\end{itemize}
For any such word, let $\word{B_b}$ be the subword of length $i + 1$ bounding $\word{v} : x$ and let $\word{B_f} : x$ be the subword bounding $S(\word{v} : x)$.
The number of subwords sharing $\word{v}:x$ equals the size of the set $\mathbf{C}(\word{w},i,r,\word{B_f}, \word{B_b}, p_f, p_b)$.
Note that as there are at most $n^2$ subwords in $\word{w}$, and checking the above conditions takes at most $O(n)$ time.
Therefore the number of such words can be computed in $O(n^3)$ time if $\mathbf{C}(\word{w},i,r,\word{B_f}, \word{B_b}, p_f, p_b)$ has been precomputed for every $i,p_b,p_f \in [m]$ and $\word{B_f}, \word{B_b} \sqsubseteq_i \word{w}$.
Hence the size of $\gamma(\word{w},m,r)$ can be computed in $O(n^5)$ time.
\end{proof}

\begin{theorem}
\label{thm:enclosing}
Let $RankEN(\word{w},m)$ be the rank of $\word{w}$ within the set of necklaces of length $m$ which enclose $\word{w} \in \Sigma^n$.
The value of $RankEN(\word{w},n)$ can be computed in $O(n^6 \log n)$ time for any $m \leq n$.
\end{theorem}

\begin{proof}
Let $\mathbf{EN}(\word{w},m)$ be the set of words belonging to a necklace of length $n$ with that enclose $\word{w}$.
By definition, observe that $|\mathbf{EN}(\word{w},m)| = \sum\limits_{r \in [m]} \gamma(\word{w},m,r)$.
Following Lemma \ref{lem:counting_gamma}, the size of $\gamma(\word{w},m,r)$ can be computed in at most $O(n^5)$ time.
Therefore the size of $|\mathbf{EN}(\word{w},m)|$ can be computed in $O(n^6 \log n)$ time.
From Lemmas \ref{lem:EL_from_words}, and \ref{lem:ranking_enclosing_words}, the value of $RankEN(\word{w},n)$ is given by the equation:
$$\sum\limits_{m | n} \sum\limits_{d | m} \mu\left(\frac{m}{d}\right) |\mathbf{EN}(\word{w},d)|$$

As there are at most $\log n$ factors of $n$, the size of $\mathbf{EN}(\word{w},d)$ needs to be evaluated for at most $O(\log n)$ values of $d$.
Therefore, a total of $O(n^6 \log n)$ time is needed to compute the size of $\mathbf{EN}(\word{w},d)$ for every factor $d$ of $n$.
Once the size of $\mathbf{EN}(\word{w},d)$ has been computed for every factor $d$ of $n$, the above equation can be evaluated in $O(\log^2 n)$ time, giving a total time complexity of $O(n^6 \log n)$.
\end{proof}




\bibliography{main.bib}
\bibliographystyle{plain}



\label{app:symmetric}

\end{document}